%% file: paper.tex
\documentclass{amsart}
\usepackage{amsmath, amsthm, amssymb}%Amer. Math. Soc. Symbols
\usepackage{latexsym}
\usepackage{dsfont}	%Doppelbalkenbuchstaben \mathds{}
\usepackage{graphicx}
\usepackage{verbatim}%F�r lange Kommentare mit begin comment
\usepackage[latin1]{inputenc} % Einfache Eingabe von Umlauten
\usepackage[T1]{fontenc} % Trennung von Umlauten enthaltenden W�tern (wenn im PDF die Schrift bel aussieht:) 
\usepackage{mathrsfs}
\theoremstyle{plain}
\newtheorem{prop}{Proposition}[section]
\newtheorem{lemma}[prop]{Lemma}
\newtheorem{thm}[prop]{Theorem}

\newtheorem{rem}[prop]{Remark}
\newtheorem*{cond}{Condition}

\theoremstyle{definition}
\newtheorem{defi}{Definition}
\newcommand{\abs}[1]{\ensuremath{\left|#1\right|}}
\newcommand{\beq}{\begin{equation}}
\newcommand{\eeq}{\end{equation}}
\newcommand{\beqs}{\begin{equation*}}
\newcommand{\eeqs}{\end{equation*}}

\newcommand{\na}{\nabla}
\newcommand{\R}{\mathbb{R}}
\newcommand{\N}{\mathbb{N}}

\newcommand{\ov}{\overline}
\newcommand{\bea}{\begin{equation}\begin{array}{rcl}}
\newcommand{\eea}{\end{array}\end{equation}}
\newcommand{\beas}{\begin{equation*}\begin{array}{rcl}}
\newcommand{\eeas}{\end{array}\end{equation*}}
\newcommand{\norm}[1]{\ensuremath{\left\|#1\right\|}}
\def\eps{\varepsilon}
\def\phi{\varphi}
\begin{document}
\title{Selfsimilar expanders of the harmonic map flow}
\author{Pierre Germain and Melanie Rupflin}
\begin{abstract} We study the existence, uniqueness, and stability of self-similar expanders of the harmonic map heat flow in equivariant settings. We show that there exist selfsimilar solutions to any admissiable initial data and that their uniqueness and stability properties are essentially determined by the energy-minimising properties of the so called equator maps.
\end{abstract}
\date{\today}
\maketitle
\input{part12}
\input{part3}
\bibliography{references}
\bibliographystyle{plain}	
\end{document}

%% file: part12.tex
%\tableofcontents

\section{Introduction}

\subsection{The harmonic map heat flow and its solutions}
The harmonic map heat flow is defined as the negative gradient flow of the Dirichlet energy of maps between manifolds. For a map $u(x,t)$ from $\mathbb{R}^d\times [0,\infty) $ to a manifold $N$, which we see as embedded in some Euclidean space with second fundamental form $\Gamma$, this equation reads
$$
\left\{ \begin{array}{l}
\partial_tu-\Delta u=\Gamma(u)(\na u,\na u) \quad \text{ on } \R^d\times [0,\infty) \\
u(t=0) = u_0. \end{array} \right.
$$
Choosing $u_0 \in H^1$ (finite energy data), Struwe~\cite{Struwe85}, Chen~\cite{Chen} (see also Rubinstein, Sternberg and Keller~\cite{RSK}), and Chen-Struwe~\cite{Chen-Struwe} were able to build up weak solutions. In the critical dimension $d=2$ the question of uniqueness of weak solutions has been analysed by Freire \cite{Freire}, Topping \cite{Topping}, Bertsch, dal Passo and Van der Hout \cite{Bertsch} and the second author of this paper \cite{Rupflin}. On the other hand, the question of  uniqueness is still open in the supercritical dimensions $d\geq 3$ that we consider here. On the one hand, examples of non-uniqueness have been obtained by Coron \cite{Coron} and Hong \cite{Hong}. On the other hand uniqueness can be obtained by working at the scaling of the equation: Koch and Lamm~\cite{Koch-Lamm} proved local well-posedness for data which are close in $L^\infty$ to a uniformly continuous map; Wang~\cite{Wang2} obtained local well-posedness for data small enough in $BMO$; finally Lin and Wang~\cite{Lin-Wang} showed uniqueness in $\mathcal{C}([0,T],W^{1,n})$.

\subsection{Equivariant setting}
We shall assume that the target manifold is spherically symmetric, more precisely that it admits coordinates $(s,\omega) \in \mathbb{R} \times \mathbb{S}^{n-1}$ in which its metric reads  $ds^2 + g^2(s)d\omega^2$. We shall further assume that the solution map is equivariant, namely in these coordinates $u(t,x) = \left( h(t,|x|), \chi\left( \frac{x}{|x|} \right) \right)$, where $\chi$ is a $k$-eigenmap, see section~\ref{sotr} for the details. Then the above equation reduces to a scalar one:
$$ 
\left\{ \begin{array}{l}
h_t-h_{rr}-\frac{d-1}{r}h_r+\frac{k}{r^2}[gg'](h)=0 \\
h(t=0) = h_0.
\end{array} \right.
$$
The archetype of such a situation is of course $N$ being the $d$ dimensional sphere, in which case $g = \sin$, $k=1$, $\chi = Id$, and the ansatz reads $u(t,x) = \left( h(t,|x|), \frac{x}{|x|} \right)$. The equator of the sphere corresponds to the solution $h \equiv \frac{\pi}{2}$; it is a trivial solution of the harmonic map heat flow. In our more general equivariant framework, an equator of a rotationally symmetric manifold is a lateral sphere of $N$ with locally maximal diameter; it corresponds to the constant in time solution of the harmonic map flow given by $h \equiv s^\star$, $s^\star$ a local maximum of $g^2$. 

\subsection{Obtained results: existence and uniqueness of self-similar solutions}
We investigate the above equation with data of the type $h \equiv s\in\R$; the expected solutions are self-similar, i.e.~of the type 
$$
h(x,t) = \psi \left( \frac{|x|}{\sqrt{t}} \right).
$$
We first establish (in theorem~\ref{THM1}) the existence of such a self-similar profile for any $s$. The next question is that of uniqueness; roughly speaking, we are able to prove the equivalence of the two following statements (see theorems~\ref{THM2} and~\ref{THM1} for the details):
\begin{itemize}
\item For any given $s$, there exists a unique self-similar profile.
\item The equator map $h \equiv s^\star$ minimises the Dirichlet energy on the unit ball among all functions in the same equivariance class and with prescribed value $h = s^\star$ on the boundary of the ball.
\end{itemize}
This equivalence stated above can be established either by ODE, or variational methods; we follow both paths, which yield complementary results. We would like to mention that parts of the above result were known to Angenent, Ilmanen and Velazquez (unpublished work, announced in \cite{Ilmanen}). Also, Biernat and Bizon~\cite{BBizon} obtained numerical and analytical results for the above problem.

\subsection{Implications for the uniqueness of solutions to the Cauchy problem} The self-similar solutions we consider are (locally in space) of finite energy; actually, they barely miss the conditions for which uniqueness or local well-posedness was stated above, thus proving the optimality of our results. 

Another non-uniqueness result for harmonic maps from $\mathbb{R}^3$ to the sphere is due to Coron~\cite{Coron}, see also Hong~\cite{Hong}. These arguments are
are more indirect and lead to only two (genuinely) different solutions, as opposed to ours, which yield a precise description of the non-unique solutions and a large number of genuinely different solutions. Though Coron's approach is very different from ours, both, interestingly enough, rely on the energy-minimising properties of certain harmonic maps.

Lastly, incoming self-similar solutions $u(x,t)=v(\frac{x}{\sqrt{-t}})$, $t<0$ have drawn a lot of attention, since they provide instances of singularity formation, or blow up, from smooth data: see Ilmanen~\cite{Ilmanen}, Fan~\cite{Fan} and Gastel~\cite{Gastel}. A combination of their results and ours yields, in some cases, non-unique continuations after the blow up time. Biernat and Bizon~\cite{BBizon} studied the question of continuation if the blow up forms along a certain profile which is numerically stable; 
they gathered analytic and numerical evidence for unique continuation in that case.

\subsection{Related results: wave maps and nonlinear heat equation}
A result similar to the one above is known for the wave map equation: see Shatah~\cite{Shatah}, Cazenave, Shatah and Tahvildar Zadeh~\cite{CSTZ}, and the first author of the present article~\cite{Germain}.

The equivalence stated above is also reminiscent of the situation for the nonlinear heat equation with power nonlinearity:
$$
\left\{ \begin{array}{l}
\partial_t v-\Delta v=|v|^\alpha v \quad \text{ on } \R^d\times [0,\infty) \\
v(t=0) = v_0. \end{array} \right.
$$
In the supercritical range, i.e.~ for $\alpha > \frac{4}{d-2}$, the equation with self-similar data $v_0 = \frac{\ell}{|x|^\alpha}$ yields self-similar solutions $v(t,x) = t^{-1/\alpha} \psi \left( \frac{r}{\sqrt t} \right)$. With respect to the issue of uniqueness it turns out that there are deep analogies between this equation and the harmonic map heat flow: the analog of the equator map is the stationary solution $\frac{\beta}{|x|^{2/\alpha}}$, with $\beta = \left( \frac{2}{\alpha} \left( d-2-\frac{2}{d} \right) \right)^{1/\alpha}$, and it is stable if and only if self-similar solutions are unique: this is the case if $\alpha > \frac{4}{d-4-2 \sqrt{d-1}}$. For this and related results we refer to \cite{HW, Weissler, PTW, EK, GV, SW, Naito, Naito2}.

\subsection{Obtained results: stability of self-similar solutions}
In theorems~\ref{chien},~\ref{thm:lin-stab}, and~\ref{THM3}, we examine the stability of our self-similar solutions, with respect to small perturbations of the data at time 0, and at time 1. We are not able to give a complete picture, but we can characterise to a large extent stable and unstable settings. The methods employed are spectral (in particular the analysis of Sturm-Liouville problems) for the linearised problem, and resort to nonlinear analysis for the full equation.

\section{Statement of the results}

\label{sotr}

\subsection{The problem under study}
We consider selfsimilar weak solutions of the harmonic map heat flow 
\begin{equation}
\label{HF}
\partial_tu-\Delta u=\Gamma(u)(\na u,\na u) \quad \text{ on } \R^d\times [0,\infty)
\end{equation}
from Euclidean space $\R^d$ into a smooth target manifold $N$.

\bigskip
 We focus here on expanding selfsimilar solutions
$$u(x,t)=v(\textstyle{\frac{x}{\sqrt{t}}}),\quad t>0,\, x\in\R^d$$
for a suitable map $v:\R^d\to N$. By the translation invariance of \eqref{HF} these maps represent all solutions of \eqref{HF} which are selfsimilar in forward time up to translations in space-time. Such solutions in the natural energy-space 

\begin{equation}
 \label{space}
(u,u_t)\in L_{loc}^\infty([0,\infty))\dot{H}_{loc}^1(\R^d)\times L^2_{loc}(\R^d\times [0,\infty))
\end{equation}
of \eqref{HF} exist only in supercritical dimensions $d\geq 3$.

These self-similar maps correspond to data which are homogeneous of degree 0
\begin{equation}
\label{IV}
u(t=0)(x)=u_0(\textstyle{\frac{x}{\abs{x}}}),\quad x\in\R^d\setminus\{0\}.
\end{equation}
Our aim in the present article will be to understand the existence, uniqueness, and stability properties of the Cauchy problem for \eqref{HF} with homogeneous data.

\subsection{Geometric setting}
\label{setting}
We consider maps from a fixed Euclidean space $\R^d$, $d\geq 3$, into a smooth rotationally symmetric target manifold $N^n$ without boundary. We introduce coordinates $(s,\omega) \in \R\times \mathbb{S}^{n-1}$ on $N$ in which the metric is given by
\begin{equation*}
%\label{metric}
ds^2+g^2(s)d\omega_{n-1}^2.
\end{equation*}
Here $d\omega_{n-1}^2$ denotes the standard metric of the sphere $\mathbb{S}^{n-1}$ and $g$ shall be a smooth function, symmetric with respect to each point $p$ where $g(p)=0$. For these special values $p$ of the lateral coordinate which represent the poles of $N$, it is necessary to assume that $\abs{g'(p)}=1$ in order to obtain a smooth manifold. The coordinate $s$ and the function $g$ are of course periodic if $N$ is compact.

Observe that the (intrinsic) diameter of the \textit{lateral sphere}
$$C_s:=\{(s,\omega):\, \omega\in \mathbb{S}^{d-1}\},\quad s\in\R$$ is equal to $\pi\abs{g(s)}$. 
We therefore call $C_{s^\star}$ an \textit{equator} of $N$ if $s^\star$ is a local maximum of $g^2$. Similarly, we call a lateral sphere whose diameter is locally minimal but positive a \textit{minimal sphere}.

We consider for the moment both compact and non-compact target manifolds $N$, but we want to assume throughout this work that 
\begin{equation}
 \sup_{s\in \R} \abs{\frac{d^2}{ds^2}(g^2)(s)}+\frac{g^2(s)}{1+s^2}<\infty.
\label{g}
\end{equation}
For simplicity, we also exclude targets for which $g'$ has roots with multiplicity greater than one or for which the function $s\mapsto \frac{d^2}{ds^2}(g^2)(s)$ is constant on an interval of positive length.
\\

We consider maps from $\R^d$ to $N$ with the following type of symmetry.

\begin{defi} Let $d, n\in \N$.
\begin{enumerate}
 \item[(i)] We call a map $\chi:\mathbb{S}^{d-1}\to \mathbb{S}^{n-1}$ a \textit{($k$-)eigenmap}, if $\chi$ is an eigenfunction of the negative Laplacian $-\Delta_{\mathbb{S}^{d-1}}$ with constant energy density
\begin{equation*}
%\label{Eig}
\abs{\na\chi}^2=k.
\end{equation*}
\item[(ii)] Let $N^n$ be a rotationally symmetric manifold and let $\chi:\mathbb{S}^{d-1}\to \mathbb{S}^{n-1}$ be an eigenmap. 
We say that a map $u:\R^d\to N^n$ is \textit{$\chi$-equivariant} if there exists a function $h:[0,\infty)\to \R$ such that 
$$u(x)=R_\chi h(x):=(h(\abs{x}), \chi(\textstyle{\frac{x}{\abs{x}}}))$$
with respect to the rotationally symmetric coordinates introduced on $N$.
\end{enumerate}
\end{defi}

The equation~\eqref{HF} becomes in equivariant coordinates
\begin{equation}
\label{equiv}
h_t-h_{rr}-\frac{d-1}{r}h_r+\frac{k}{r^2}G(h)=0
\end{equation}
(where $G:=gg'$), see Lemma~\ref{lemma:equi}.

Let us remark that the spectrum of the negative Laplacian on the sphere $\mathbb{S}^{d-1}$
$$\{l(d-2+l):\, l\in\N\}$$
contains no eigenvalues smaller than $d-1$. An example of a $(d-1)$-eigenmap is of course the identity $id:\mathbb{S}^{d-1}\to \mathbb{S}^{d-1}$ with the corresponding equivariant maps being the corotational maps $x\mapsto (h(\abs{x}),\frac{x}{\abs{x}})$.
The components of general eigenmaps with eigenvalue $\lambda_l=l(d-2+l)$ are given by the restriction of $l$-homogeneous, harmonic polynomials to the sphere, see \cite{Eells+Rato}, chapter VIII.
\\

\subsection{Equator maps and their minimising properties}

Given any equator $C_{s^\star}$ of $N$ and any eigenmap $\chi$, we define the corresponding \textit{equator map} by
$$u^\star=u_{\chi,s^\star}^\star:=R_\chi h^\star $$
for the constant function $h^\star\equiv s^\star$.
Note that this equator map and its properties depend both on the eigenmap $\chi$ and on the value of $s^\star$.

\begin{defi}
\label{def1}
Let $C_{s^\star}$ be an equator of a rotationally symmetric manifold $N$ and let $\chi$ be an eigenmap. 
We say that the equator map $u_{\chi,s^\star}^\star$ is \textit{$\chi$-energy-minimising} if it minimises the Dirichlet energy 
$$E(u,B_1(0))=\frac12 \int_{B_1(0)}\abs{\na u}^2 dx$$
in the set  
\begin{equation*}
%\label{F}
\mathcal F_{\chi,s^\star}:=\{R_\chi h:\,\, h:[0,1]\to \R \text{ with } h(1)=s^\star\} 
\end{equation*}
of $\chi$-equivariant functions with the same boundary data. 
\end{defi}
Notice that we do not demand that the equator map $u_{\chi,s^\star}^\star$ be energy-minimising in the larger class of maps 
$$F=\{u\in H^1(B_1,N):\, u\vert_{\partial B_1}=u_{\chi,s^\star}^\star\vert_{\partial B_1}\}.$$ 
We cannot exclude the possibility of symmetry breaking in the sense that 
$$\inf_{v\in F} E(v,B_1)< \inf_{v\in \mathcal F_{\chi,s^\star}} E(v,B_1).$$
An example for such an occurrence in a related context of $G$-equivariant harmonic map was given by Gastel \cite{Gastel-G-harmonic} based on the analysis of singularities by Brezis, Coron and Lieb \cite{Brezis-Coron-Lieb}. 
\\

The following proposition provides a simple criterion to test whether or not a given equator map is $\chi$-energy-minimising.
\begin{prop}
\label{Prop3.1}
Let $d\geq 3$, let $N^n$ be a smooth, rotationally symmetric manifold and let $\chi:\mathbb{S}^{d-1}\to \mathbb{S}^{n-1}$ be a $k$-eigenmap. Let $C_{s^\star}$ be an equator of $N$ and recall $G:=g\cdot g'$.
\begin{enumerate}
\item[(i)] If
\begin{equation}
-4kG'(s^\star) < (d-2)^2,
\end{equation}
the equator map is \underline{locally} (i.e. for small perturbations) $\chi$-energy-minimising.
\item[(ii)] If 
\begin{equation}
\label{Theta}
-4kG'(s^\star)> (d-2)^2,
\end{equation}
the equator map $u_{\chi,s^\star}^\star$ is not (locally) $\chi$-energy-minimising.
\item[(iii)] Suppose that 
\begin{equation*}
%\label{CON}
-4kG'(s)\leq (d-2)^2 \text{ for } s\in[s^\star-S,s^\star+S]
\end{equation*}
where $S:=\frac{2\sqrt{k}}{d-2}\cdot \norm{g}_\infty$. Then $u_{\chi,s^\star}^\star$ is \underline{globally} $\chi$-energy-minimising.
\end{enumerate}
\end{prop}
Applying the above criterion to the case where the target manifold is the sphere $\mathbb{S}^{d}$ with the standard metric $ds^2 + \sin(s)^2d\omega_{d-1}^2$, and the maps are corotational $\chi=Id$, gives the well-known result: the equator map $R_{\chi} \frac{\pi}{2}$ is energy-minimising if and only if $d\geq 7$.

\subsection{Existence and uniqueness results}

We are able to prove existence of solutions to \eqref{HF} for homogeneous data in essentially all cases; notice that we are dealing with infinite energy solutions, thus the existence theorems by Chen \cite{Chen} and Chen and Struwe \cite{Chen-Struwe} do not apply here.
The question of uniqueness is much more interesting; roughly speaking, we shall prove that solutions to the Cauchy problem for \eqref{HF} with homogeneous data are unique if and only if the equator map is energy-minimising. More precise formulations of this idea are contained in the two following theorems.

\bigskip

We begin with a very general setting $(N,\chi)$, where the equator map is not $\chi$-energy-minimising. 
\begin{thm}
\label{THM2}
Let $d\geq 3$, let $N^n$ be a rotationally symmetric manifold such that \eqref{g} is satisfied and let $\chi:\mathbb{S}^{d-1}\to \mathbb{S}^{n-1}$ be a fixed eigenmap. Assume that $N$ has an equator map $u_{\chi, s^\star}^\star$ which is not $\chi$-energy-minimising.

Then there exists a selfsimilar and $\chi$-equivariant weak solution $u\in H_{loc}^1(\R^d\times [0,\infty))$ of the initial value problem \eqref{HF}, \eqref{IV} that is not constant in time for the initial data $u_0=u_{\chi, s^\star}^\star$.
\end{thm}

We shall next impose the following restrictions on the function $g$ representing the metric of $N$. 

\begin{cond}[C1]
Let $C_{s^\star}$ be an equator of a compact, rotationally symmetric manifold $N$ and let $s_1<s^\star<s_2$ be the local minima of $g^2$ to the left and to the right of $s^\star$, i.e.~the local minima of $g^2$ such that $g^2\vert_{[s_1,s^\star]}$ is increasing while $g^2\vert_{[s^\star,s_2]}$ is decreasing.

We then demand that
\begin{equation*}
G'(s^\star)=\min_{s\in[s_1,s_2]} G'(s)
\end{equation*} 
(recall $G(s):=g'(s)g(s)$).
\end{cond}

For manifolds that contain a minimal sphere $C_{s_0}$ we furthermore impose 

\begin{cond}[C2]
Let $k$ be any given eigenvalue of $-\Delta_{\mathbb{S}^{d-1}}$. We say that a rotationally symmetric manifold $N$ fulfils condition (C2) (for $k$) if for each minimal sphere $C_{s_0}$ of $N$ 
$$G'(s_0)\geq \frac{d-1}{k}.$$ 
\end{cond}

Conditions (C1) and (C2) are fulfilled for a wide variety of rotationally symmetric manifolds, in particular for round spheres and for rotationally symmetric ellipsoids. 
\\

\begin{thm}
\label{THM1}
Let $d\geq 3$, let $N^n$ be a compact, rotationally symmetric manifold and let $\chi:\mathbb{S}^{d-1}\to \mathbb{S}^{n-1}$ be an eigenmap.
\begin{enumerate}
\item[(i)] There exists a selfsimilar and equivariant weak solution of \eqref{HF} for any admissible initial data, i.e. for every map $u_0(x)=(s,\chi(\frac{x}{\abs{x}}))$, $s\in \R$. 
\item[(ii)] Assume that all equator maps of the manifold $M$ are $\chi$-energy minimising and that conditions (C1) and (C2) are satisfied. Then the solution of (i) is unique in the class of all equivariant and selfsimilar weak solutions.
\item[(iii)] Assume that the manifold $N$ has an equator $C_{s^\star}$ such that $$-4kG'(s^\star)>(d-2)^2,$$ i.e.~such that the corresponding equator map is not even locally energy minimising. Then given any number $K\in\N$, there exists a neighbourhood $U_K$ of $s^\star$ such that the initial value problem \eqref{HF}, \eqref{IV} has at least $K$ different weak solutions which are $\chi$-equivariant and selfsimilar for each initial data $u_0(x)=(s,\frac{x}{\abs{x}})$ with $s\in U_K$.
\end{enumerate}
\end{thm}

\begin{rem}\label{rem:mon}
All solutions obtained in theorems~\ref{THM2} and~\ref{THM1} satisfy the monotonicity formula of Struwe \cite{Struwehighdim}. For (constant in time) equivariant harmonic maps this follows since the maps are stationary harmonic (see Lemma 7.4.1 in Lin and Wang~\cite{LW}). For more general selfsimilar solutions the monotonicity formula can be shown using the asymptotics of solutions of \eqref{EX}.
\end{rem}

\subsection{Stability at time $t=1$}

In the previous section, we characterised precisely the existence and uniqueness properties of self-similar solutions to the harmonic map heat flow. Our aim now is to study their stability: we focus in this section on the effect of a perturbation at time $t=1$, and in the next on a perturbation occurring at time $t=0$.

Let $\psi$ be one of the self-similar profiles whose existence has been established, and consider a perturbation $u=R_\chi \left[\psi\left(\frac{\cdot}{\sqrt{t}}\right) + f\right]$. For data given at $t=1$, the Cauchy problem becomes
\begin{equation}
\label{ecureuil}
\left\{
\begin{array}{l}
f_t-f_{rr}-\frac{d-1}{r}f_r+\frac{k}{r^2}\left[ G\left( f+ \psi\left(\frac{\cdot}{\sqrt{t}}\right)\right) - G\left( f+ \psi\left(\frac{\cdot}{\sqrt{t}}\right)\right) \right]=0 \\ f(t=0)=f_0 
\end{array} 
\right. .
\end{equation}
By scaling invariance, it is of course equivalent to study the problem from $t=1$ or any other positive time.

Suppose first that $\psi$ is provided by Theorem~\ref{THM2}. The proof given in section~\ref{section:Non-unique} will give that $\psi$ minimises the functional 
$$
\overline{E}(f):=\int_0^\infty \big[\abs{f'}^2+\frac{k}{r^2}\big(g^2(s^\star+f)-g^2(s^\star)\big)\big]\,r^{d-1}e^{r^2/4}\,dr,
$$
which is very reminiscent of the well-known monotonicity formula for the harmonic map heat flow~\cite{Struwehighdim}. Thus our first stability result essentially corresponds to a forward time version of the monotonicity formula.

\begin{thm}
\label{chien}
Let $\psi$ be given by Theorem~\ref{THM2}, and let $f$ solve~(\ref{ecureuil}). Then $\overline{E}\left( f(\sqrt{t} \cdot) + \psi \right)$ is a decreasing function of time.
\end{thm}

Even though $\overline{E}$ is only minimised at $\psi$, it is not clear to us to what extent $\overline{E}\left( f(\sqrt{t} \cdot) + \psi \right)$ controls $f$.

\bigskip

Consider now a general profile $\psi$, given by Theorem~\ref{THM1}; we want to investigate linear stability. The linearised version of~(\ref{ecureuil}) reads
\begin{equation}
\label{LIN}
\left\{
\begin{array}{l}
f_t-f_{rr}-\frac{d-1}{r}f_r+\frac{k}{r^2} G'\left( \psi\left(\frac{\cdot}{\sqrt{t}}\right)\right) f=0 \\ f(t=1)=f_0 
\end{array} 
\right. .
\end{equation}

A spectral analysis of the above problem in self-similar variables will lead to the following theorem.
\begin{thm}
\label{thm:lin-stab}
Let $\psi$ be given by Theorem~\ref{THM1}, and $v$ solve~(\ref{LIN}).
\begin{itemize}
\item[(i)] If $\psi$ is monotone,
$$
\|v(t)\|_{L^2(e^{\frac{r^2}{4t}}r^{d-1}dr)} \leq t^{\frac{d}{2}-2} \|v_0\|_{L^2(e^{\frac{r^2}{4}}r^{d-1}dr)}.
$$
In particular, if $d\leq 4$, $\displaystyle \|v\|_{L^2(e^{\frac{r^2}{4t}}r^{d-1}dr)}$ is decreasing.

\item[(ii)] If $\psi$ has $K$ local extrema, there exists $\gamma>0$ and $K$ linearly independent data $v_0$ such that
$$
\|v\|_{L^2(e^{\frac{r^2}{4t}}r^{d-1}dr)} \gtrsim t^{\gamma + \frac{d}{2}-2}.
$$
In particular, if $d \geq 4$, this corresponds to a growing norm.
\end{itemize}
\end{thm}

\subsection{Stability at time $t=0$}

We focus in this section on the effect of a perturbation at time $t=0$.

We start with the most simple type of self-similar solutions: the maps which are constant in time, mapping $\R^d$ onto an equator or a minimal sphere. If $u = R_\chi (f+s^\star)$, the equation under study is
\begin{equation}\label{perturbed}
\left\{
\begin{array}{l}
f_t-f_{rr}-\frac{d-1}{r}f_r+\frac{k}{r^2}\left[G(f+s^\star)-G(s^\star)\right]]=0 \\ f(t=0)=f_0 
\end{array} 
\right.
\end{equation}

\begin{thm} 
\label{THM3}
Let $d\geq 3$, let $N^n$ be a rotationally symmetric manifold such that \eqref{g} is satisfied and let $\chi:\mathbb{S}^{d-1}\to \mathbb{S}^{n-1}$ be a fixed eigenmap. Suppose $s^\star$ is such that $G(s^\star)=0$; it corresponds to a constant solution of~\eqref{HF} given by $u_{\chi, s^\star}^\star$. Consider the perturbed equation~(\ref{perturbed}).
\begin{itemize}
\item[(i)] If
$$
kG'(s^\star) > -\frac{(d-2)^2}{4},
$$
this equation is linearly stable (in more precise terms: the Cauchy problem associated with the linear part of equation \eqref{perturbed} is globally well-posed in $L^2$, and the $L^2$ norm is decreasing).
\item[(ii)] If
$$
k \displaystyle \inf_{\R} G' > -\frac{(d-2)^2}{4},
$$
there exists a global weak solution $f$ to the above equation, satisfying 
\begin{equation}
\label{tomahawk}
\displaystyle \|f\|_{L^\infty ([0,\infty),L^2(\mathbb{R}^d))}^2 + \|\nabla f\|_{L^2 ([0,\infty), L^2(\mathbb{R}^d))}^2 \lesssim \|f_0\|_2^2.
\end{equation}
\item[(iii)] If
$$G'(s^\star)>0,$$
i.e.~if $s^\star$ is the coordinate of a pole or a minimal sphere, the above equation is globally well-posed in $L^\infty$ for small data
 (in more precise terms: if $f_0$ is small in $L^\infty$, there exists a solution $f$ of \eqref{perturbed} in $L^\infty([0,\infty), L^\infty)$, which is unique in a small enough ball and depends continuously on $f_0$).
\end{itemize}
\end{thm}

\begin{rem} \begin{enumerate} \item The weak solutions in (ii) do not share - even locally - the same functional setup as the Struwe solutions: whereas the former give $f \in L^\infty_t L^2$, the latter would roughly correspond to $f \in L^\infty_t \dot{H}^1$.

\item Notice that the spaces (for the data as well as the solution) for which well-posedness is proved in (ii) and (iii) are at the same scaling as the equation: their norms are invariant by the scaling which leaves the equation invariant. We do not claim any optimality for these spaces: there should exist larger spaces in which the equation is well-posed.

\item The above theorem gives sufficient conditions on $G'(s^\star)$ for various kinds of stability results to hold true. We ask in Subsection~\ref{instab} whether they are also necessary. The answer is shown to be yes for (i) and (ii).
\end{enumerate}
\end{rem}
For non-constant selfsimilar solutions we obtain the following stability result
\begin{thm}
\label{THM4}
In the setting which was just described, consider the perturbed equation for a self-similar profile: If $u(t) =R_\chi \left[\psi\left(\frac{\cdot}{\sqrt{t}}\right) + f\right]$ solves \eqref{HF} then $f$ solves
$$
\left\{
\begin{array}{l}
f_t-f_{rr}-\frac{d-1}{r}f_r+\frac{k}{r^2}\left[ G\left( f+ \psi\left(\frac{\cdot}{\sqrt{t}}\right)\right) - G\left( \psi\left(\frac{\cdot}{\sqrt{t}}\right)\right) \right]=0 \\ f(t=0)=f_0 
\end{array} 
\right. .
$$
Suppose that $\psi$ is such that $G'(\psi(r)) > 0$ for all $r>0$. Then the above equation is globally well-posed in $L^\infty$, namely, for $f_0$ small enough in $L^\infty$, there exists a solution $u$ in $L^\infty_t L^\infty$, which is unique in a small enough ball, and depends continuously on the data.
\end{thm}

This theorem follows by the same arguments as Theorem~\ref{THM3}, (iii), thus we skip its proof.
\\

\textbf{Notation}
The notation $A \lesssim B$ means: there exists a constant $C$ such that $A \leq CB$. \\

\textbf{Acknowledgement} The research of the second author was partially supported by the Swiss National Science Foundation and The Leverhulme Trust.

\section{Preliminaries}

\subsection{Weak solutions of the harmonic map flow in the equivariant setting}
\label{section:equiv}
Let $N^n$ be a rotationally symmetric manifold, let $g:\R\to\R$ be the function describing the metric of $N$ and let 
$\chi:S^{d-1}\to S^{n-1}$ be an eigenmap to eigenvalue $k\in\N$. A short calculation shows that the Dirichlet energy of an equivariant map $v=R_\chi h$ is given by 
$$E(v,B_R(0))=\frac12\int_{B_R(0)}\abs{\na v}^2 dx=\frac{c_d}2\int_0^R[\abs{h'}^2+\frac k{r^2}g^2(h)]r^{d-1}\, dr$$
for $c_d=\abs{S^{d-1}}$, the Hausdorff-measure of the $d-1$ dimensional unit sphere. 

In view of condition \eqref{g} the set of functions $h$ which induce equivariant maps with locally finite energy can be described by
\begin{defi}
Given $d\in\N$ and a ball $B_R=B_R(0)\subset \R^d$ we define
$$H^1_{rad}(B_R):=\{h:[0,R]\to \R: \, \int_0^R(\abs{h'}^2+\frac{h^2}{r^2} r^{d-1} \,dr<\infty\},$$
and set 
$$H_{rad}^1(\R^d):=\bigcap\limits_{R>0} H^1_{rad}(B_R).$$
\end{defi}
Observe that the equivariant function $R_\chi h:\R^d\to N$ is an element of $H_{loc}^1(\R^d)$ but not necessarily of $H^1(\R^d)$ if $h\in H_{rad}^1(\R^d)$. Let us also remark that the global energies $E(u(t),\R^d)$ of solutions of the harmonic map heat flow \eqref{HF} are in general infinite.
\\

Direct computations (see e.g.~\cite{Germain}) lead to the following characterisation of equivariant weak solutions of the harmonic map heat flow.
\begin{lemma}\label{lemma:equi} Consider a rotationally symmetric manifold $N^n$ with metric described by $g\in C^1(\R)$ and let $\chi:S^{d-1}\to S^{n-1}$ be a $k$-eigenmap.
\begin{enumerate}
\item[(i)] Let $u$ be an element of the energy-space \eqref{space} of the form $u=R_{\chi} h$ for a function $h:\R_0^+\times \R_0^+\to \R$. Then $u$ is a weak solution of \eqref{HF} if and only if $h$ solves the scalar partial differential equation
\begin{equation}
\label{PDE1}
h_t-[h_{rr}+\frac{d-1}{r}h_r-\frac{k}{r^2}G(h)]=0 \text{ on }\R_0^+\times \R_0^+
\end{equation}
in the sense of distributions.
\item[(ii)] Let $u$ be an element of the energy-space \eqref{space} that is of the form $u(x,t)=R_\chi h(\frac{x}{\sqrt{t}}),\, t>0$, for some  $h:\R_0^+\to \R$.\\
Then $u$ is a weak solution of \eqref{HF} if and only if $h$ solves the differential equation 
\begin{equation}
h''+\big(\frac{d-1}r+\frac r2\big)h'-\frac k{r^2}G(h)=0 \text{ on } (0,\infty).
\label{EX}
\end{equation}
\end{enumerate}
\end{lemma}
Remark that we can rewrite equation \eqref{EX} in divergence-form as
\begin{equation}
 \label{div-form}
\frac{d}{dr}\Big(r^{d-1}e^{r^2/4} h'(r)\Big)=kr^{d-3}e^{r^2/4}G(h).
\end{equation}
\begin{comment}
and immediately obtain that $h:(0,\infty)\to \R$ is a solution of \eqref{EX} if and only if it is a critical point of the functional 
\begin{equation}
\label{var-form}
\mathcal{E}(f)=\int_0^\infty [\abs{f'}+\frac{k}{r^2}g^2(f)] e^{r^2/4}r^{d-1}\, dr
\end{equation}
with respect to compactly supported variations.
\end{comment}
It can be easily checked that a selfsimilar map $u(x,t)=R_\chi h(\frac{x}{\sqrt{t}})$ is an element of the energy-space if and only if $h\in H_{rad}^1(\R^d)$ and
\begin{equation}
\label{2.4}
\int_0^1(\sqrt{t})^{d-4}\int_1^{R/\sqrt{t}}\abs{h'}^2 r^{d+1}\,dr\,dt<\infty.
\end{equation}
At first glance the assumption $h\in H^1_{rad}(\R^d)$ imposes only a mild constraint on the behaviour of $h$ near $r=0$ while the condition \eqref{2.4} seems to seriously restrict the allowed behaviour at infinity. We will see later that the converse is true for solutions of equation \eqref{EX}. Indeed, the first derivative of each solution of \eqref{EX} decays sufficiently fast for \eqref{2.4} to be fulfilled, but most solutions of \eqref{EX} blow up as $r\to 0$ in such a way that $h\notin H^1_{rad}(\R^d)$. 
\\

Let us finally remark that the trace of a selfsimilar map $u(x,t)=R_\chi h(\frac{x}{\sqrt{t}})$ on the time slice $\R^d\times\{0\}$ is given by $u(x,0)=(s, \chi(\frac{x}{\abs{x}}))$ if $h$ converges to $s\in\R$ as $r\to \infty$.

\subsection{Characterisation of energy-minimising equator maps}
\label{section:class}
As remarked in \cite{Germain}, the criterion given in Proposition \ref{Prop3.1} is closely related to the value of the optimal constant in the Hardy inequality.
\begin{lemma}
\label{lemma:Hardy}
Let $d\geq 3$. Then $C_H=\frac{4}{(d-2)^2}$ is the optimal constant such that the Hardy inequality 
\begin{equation}
\label{Hardy}
\int_0^1w^2 r^{d-3}\,dr\leq C_H\int_0^1\abs{w'}^2r^{d-1}\,dr
\end{equation}
holds true for all $w\in H^1_{rad}(B_1)$ with $w(1)=0$.
\end{lemma}
For a proof of this result we refer to \cite{Vazquez-Zuazua}.

\begin{proof}[Proof of Proposition \ref{Prop3.1}]
The proof of (i) and (ii) follow directly from Hardy's inequality, as can be seen in \cite{Germain}.

Let us therefore assume that 
$$kG' (s)\geq \frac{(d-2)^2}{4}=-C_H^{-1} \text{ for all } s\in [s^\star-S,s^\star+S]$$ where $S:=\sqrt{kC_H}\cdot \norm{g}_\infty =\frac{2\sqrt{k}}{d-2}\cdot \norm{g}_\infty$. Using the quadratic Taylor expansion of $g^2$ around $s^\star$, we find for these values of $s$
$$k\big[g^2(s)-g^2(s^\star)\big]\geq -C_H^{-1}(s-s^\star)^2.$$
Remark that this estimate is trivially true if $\abs{s-s^\star}\geq S$.

The above Hardy inequality thus implies that for all $h\in H_{rad}^1(B_1)$ with $h(1)=s^\star$  
\begin{align*}
\int_0^1\big[\abs{h'}^2+\frac k{r^2}\big(g^2(h)-g^2(s^\star)\big)\big]r^{d-1}\, dr&\geq\, \int_0^1\big[\abs{(h-s^\star)'}^2-\frac{C_H^{-1}}{r^2}(h-s^\star)^2\big]r^{d-1}\,dr\\
&\geq\, 0
\end{align*}
and thus that $E(R_\chi h\,B_1)\geq E(u^\star,B_1)$.
\end{proof}

\section{The variational approach: proof of Theorem \ref{THM2}}

\label{section:Non-unique}
We prove the first non-uniqueness result, Theorem \ref{THM2} by variational methods. Contrary to the arguments used for the proof of Theorem \ref{THM1}, we do not require any restrictions on the manifold $N$ other than the general assumption \eqref{g}. Theorem \ref{THM2} is thus valid also for a large class of non-compact rotationally symmetric target manifolds.

\subsection{The variational problem}

Let $N$ be a rotationally symmetric manifold, $\chi$ an eigenmap and let $C_{s^\star}$ be an equator of $N$. Assume that the equator map $u^\star_{\chi, s^\star}$ is not $\chi$-energy-minimising. According to the discussion in section \ref{section:equiv} we need to establish the existence of a non-constant solution $h\in H^1_{rad}(\R^d)$ to equation \eqref{EX} with $\lim_{r\to \infty}h(r)=s^\star$ which satisfies condition \eqref{2.4}.\\

We consider the set 
$$\mathcal F:= \{f\in H^1_{rad}(\R^d): \text{supp}(f)\subset\subset [0,\infty)\}$$
and take its closure $\ov{\mathcal{F}}$ with respect to the norm
$$\norm{f}^2:=\int (\abs{f'}^2+\frac{\abs{f}^2}{r^2})r^{d-1}e^{r^2/4}\,dr.$$

Let us remark that condition \eqref{2.4} is trivially fulfilled for elements of $\ov{\mathcal F}$ and that functions in $\ov{\mathcal F}$ converge to zero as $r\to \infty$. 
In view of the divergence form \eqref{div-form} of equation \eqref{EX} we consider the variational integral
\begin{equation}
\label{defE}
\overline{E}(f):=\int_0^\infty \big[\abs{f'}^2+\frac{k}{r^2}\big(g^2(s^\star+f)-g^2(s^\star)\big)\big]\,r^{d-1}e^{r^2/4}\,dr
\end{equation}
on the reflexive space $(\ov{\mathcal{F}},\norm{\,\cdot\,})$ ($\overline{E}$ is finite on $\mathcal{F}$ since $g'(s^\star)=0$). We prove that this functional has the following properties
\begin{enumerate}
\item $\overline{E}(\cdot)$ is weakly lower semi-continuous and bounded from below on $(\ov{\mathcal{F}},\norm{\,\cdot\,})$.
\item If the equator map $u_{s^\star,\chi}^\star$ is not $\chi$-energy-minimising, then 
$$\inf\limits_{f\in\overline{\mathcal{F}}} \overline{E}(f)<0=\overline{E}(0).$$
\end{enumerate}
We therefore find that $\overline{E}$ achieves its global minimum for a function $f\in\ov{\mathcal{F}}$ that is not identically zero. Consequently $s^\star+f$ is a non-constant solution of \eqref{EX} that induces a selfsimilar weak solution of the harmonic map flow for initial data $u_0=u_{\chi,s^\star}^\star$ different from the time-independent equator map. 

It remains to prove the above claims about $\overline{E}$. 

\subsection{Proof of claim 1}

We use that $C_1:=\sup_{s\in \R} \abs{\frac{d^2}{ds^2} g^2(s)}<\infty$ by assumption \eqref{g} and estimate
$$g^2(s^\star+f)-g^2(s^\star)\geq -\min(C_1 f^2,g^2(s^\star)).$$
Given any $R>0$ and any 
$f\in \ov{\mathcal{F}}$, we thus obtain
\begin{align*}
 \overline{E}(f)&=\int_0^\infty \big[\abs{f'}^2+\frac{k}{r^2}\big(g^2(s^\star+f)-g^2(s^\star)\big)\big]\,r^{d-1}e^{r^2/4}\,dr\\
&\geq\int_0^\infty\abs{f'}^2r^{d-1}e^{r^2/4}\,dr-kC_1\int_R^\infty f^2r^{d-3}e^{r^2/4}\, dr\\
&\qquad -kg^2(s^\star)\int_0^Rr^{d-3}e^{r^2/4}\, dr\\
&\geq\,\int_0^\infty\abs{f'}^2r^{d-1}e^{r^2/4}\,dr-CR^{-2}\int_R^\infty f^2r^{d-1}e^{r^2/4}\, dr-C(R)
\end{align*}
for a constant $C(R)$ independent of $f$.

In the weighted space $\ov{\mathcal{F}}$ the Hardy inequality 
\begin{equation}
\label{Hardy2}
\int_0^\infty f^2(1+r^2)r^{d-3}e^{r^2/4}\, dr \leq C \int_0^\infty\abs{f'}^2r^{d-1}e^{r^2/4} \,dr
\end{equation}
holds true for a universal constant $C=C(d)$, see e.g. \cite{Vazquez-Zuazua}. Choosing the number $R>0$ in the above estimate large enough, we thus obtain a uniform lower bound for $\overline{E}$ on $\ov{\mathcal{F}}$. 
\\

Remark now that inequality \eqref{Hardy2} shows furthermore that an equivalent norm to $\norm{\,\cdot\,}$ on $\ov{\mathcal{F}}$ is given by $\vert\vert\vert f \vert\vert\vert^2:= \int \abs{f'}^2r^{d-1}e^{r^2/4}\,dr$. The weak lower semi-continuity of $\ov{E}$ then follows from the estimate
$$\int_{R}^\infty (g^2(s^\star+f)-g^2(s^\star))r^{d-3} e^{r^2/8}\leq \frac{C_1}{R^2} \norm{f}^2$$ 
and the lemma of Fatou applied on finite intervals $[0,R]$. \\

\subsection{Proof of claim 2}

In order to prove property (2) for $\ov{E}$, we define a family of weighted energies $(E_\lambda)_{\lambda\in[0,1]}$ on the space $\overline{\mathcal{F}}$ by 
$$E_\lambda(f):=\int_0^\infty \big[\abs{f'}^2+\frac{k}{r^2}\big(g^2(s^\star+f)-g^2(s^\star)\big)\big]r^{d-1}e^{\lambda r^2/4}\, dr.$$
Note the scaling 
$$E_\lambda(h)=\lambda^{-\frac{d-2}{2}}\cdot E_1\big(h(\frac{\cdot}{\sqrt{\lambda}})\big)=\lambda^{-\frac{d-2}{2}}\cdot\ov{E}\big(h(\frac{\cdot}{\sqrt{\lambda}})\big).$$
Since the equator map $u^\star=u_{s^\star,\chi}^\star$ is by assumption not energy-minimising, there exists a function $h\in H_{rad}^1(B_1)$ with $h(1)=0$ and  %the Dirichlet energy 
$$E(R_\chi (s^\star+h),B_1)-E(u^\star,B_1)=\frac{c_d}{2}\int_0^1 \big[\abs{h'}^2+\frac{k}{r^2}\big(g^2(s^\star+h)-g^2(s^\star)\big)\big]r^{d-1}\, dr <0.$$ 

Extending $h$ by zero on $[1,\infty)$, we thus obtain that $E_0(h)<0$ and by continuity of $\lambda\mapsto E_{\lambda}(h)$ also $E_{\lambda}(h)<0$ for $\lambda>0$ small. Consequently 
$$\inf\limits_{f\in\overline{\mathcal F}} \ov{E}(f) = \lambda^{\frac{d-2}{2}}\inf\limits_{f\in\overline{\mathcal F}} E_\lambda(f)<0$$
as claimed.

\section{Properties of the associated ordinary differential equation}

\label{section:general}

\subsection{Existence, uniqueness, and asymptotic behaviour}
In this section we collect several important properties of solutions to the differential equation \eqref{EX} characterising selfsimilar solutions in the equivariant setting. We shall assume from now on that $N$ is compact and thus in particular that $g$, $g'$ and $g''$ are bounded periodic functions on $\R$.
\\

We first show that the behaviour of arbitrary solutions $h$ of \eqref{EX} for $r\to \infty$ can be described by
\begin{lemma}\label{lemma2} $ $
\begin{enumerate}
	\item[(i)] Let $h$ be any solution of \eqref{EX}. Then there exists a constant $C=C(h)$ such that 
	\begin{equation*}
	\abs{h'(r)}\leq \frac{C}{r^3}\text{ for all } r\geq 1.
	\end{equation*}
	\item[(ii)] This inequality holds true with a universal constant  
$\ov{C}=\ov{C}(g,k)$ for all solutions $h$ of \eqref{EX} with $\lim_{r\to 0}r\cdot h'(r)=0$.
\end{enumerate}
\end{lemma}

\begin{proof}[Proof of Lemma \ref{lemma2}]

The quantity
\begin{equation}\label{V}
V(r)=V(h)(r):=r^2\abs{h'(r)}^2-kg^2(h(r))
\end{equation}
is decreasing for any non-constant solution $h$ of \eqref{EX}
with
\begin{align}
V'(r)&=\,-r^2\abs{h'(r)}^2\big[\frac{2(d-2)}{r}+r\big].
\label{V'1}
\end{align}
The possible behaviour of $F(r):=V(r)+kg^2(h(r))=r^2\abs{h'(r)}^2$ is thus constrained by
\begin{align*}
F'(r)+rF(r)&\leq 2kG'(h)\cdot h'(r)\leq \frac{r}2F(r)+\frac{C}{r^3}.
\end{align*}
Integrating the above inequality we then obtain that 
\beq \label{estf}
F(r)\leq (e^{1/4}F(1)+C)\cdot e^{-r^2/4}+\frac{C}{r^4}\eeq
which leads to the desired estimate of statement (i). This estimate is independent of the solution $h$ if $rh'(r)\to 0$ as $r\to 0$ since in this case $0\geq V(0)\geq V(r)\geq F(r)-k\norm{g}_{\infty}^2$ for every $r>0$, and thus in particular $|F(1)| \leq k \|g\|_\infty^2$.
\end{proof}
An important consequence of Lemma \ref{lemma2} is that each solution $h$ of \eqref{EX} converges as $r\to \infty$ in such a way that condition \eqref{2.4} is satisfied. In order to find selfsimilar solutions of the harmonic map heat flow we can therefore concentrate on finding solutions of \eqref{EX} that are elements of $H^1_{rad}(\R^d)$.\\

\begin{comment}
\begin{rem}\label{rem:tang} The fact that $V(h)$ is non-increasing implies that non-constant solutions of \eqref{EX} cannot be tangential to the horizontal line $s=s^\star$ for finite values of $r>0$ if $C_{s^\star}$ is an equator of $N$. 

Furthermore, if $h\in C([0,\infty))$ is a non-constant solution of \eqref{EX} for which $\lim_{r\to 0} rh'(r)=0$ then $V(h)<-kg^2(h(0))$ on $(0,\infty)$. Consequently $h$ cannot reach points $s$ with $g^2(s)\leq g^2(h(0))$ for $r>0$. In addition, the derivative of $h$ satisfies the inequality $r\abs{h'(r)}\leq k\norm{g}_{\infty}$ on $(0,\infty)$.
\end{rem}
\end{comment}

\begin{prop}
\label{Prop6}
%\begin{enumerate}
%	\item[(i)] 
Let $s_0$ be a local minimum of $g^2$ and let $a\in\R$. 
Then there exists a solution $h_a\in C^2((0,\infty))\cap C^0([0,\infty))$ of equation \eqref{EX} such that 
\begin{equation} \label{AB}
h_a(0)=s_0\text{ and } \lim\limits_{r\to 0}r^{-\gamma}(h_a(r)-s_0)=a,
\end{equation}
where $\gamma=\frac12(\sqrt{(d-2)^2+4kG'(s_0)}-(d-2))$. Furthermore, $r^{1-\gamma}h_a'(r)\to \gamma a$ as $r\to 0$, and $h_a\in H^1_{rad}(\R^d)$.
If condition (C2) is satisfied this solution is uniquely determined by \eqref{AB}.
%\item[(ii)] The limit $\lim_{r\rightarrow \infty} h_a(r)$ exists for any $a$, and depends continuously on $a$. (Macht hier noch keinen Sinn da man auch nachfolgendes Lemma zum Beweis braucht.
%\end{enumerate}
\end{prop}
Let us remark that the solutions $(h_a)$ of \eqref{EX} constructed in Proposition \ref{Prop6} 
induce a one-parameter family of selfsimilar weak solutions of the harmonic map flow. In fact, as we will prove in section \ref{section:THM1}, the only other solutions of \eqref{EX} which induce selfsimilar weak solutions of \eqref{HF} are the constant functions $h=s^\star$, for $C_{s^\star}$ an equator of $N$. 

This proposition can be obtained by well known methods in the theory of ordinary differential equations and is presented in detail in \cite{Dis}, appendix B.1. The assumption (C2) is necessary only for the proof of the uniqueness aspect and implies that the exponent $\gamma\geq1$ This allows us to apply a boundary point lemma such as Theorem 1.4 of \cite{Protter-Weinberger} to the rescaled difference $f(r)=r^{1-\gamma}(h_1-h_2)$ of two solutions of \eqref{EX}. We obtain that if $f(0)=f'(0)=0$ then $f$ must identically vanish and thus the two solutions coincide.

A good way to analyse the behaviour of the solutions $h_a$ is to compare them with the corresponding solutions of the equation 
\begin{equation}
\label{HM}
h''+\frac{d-1}rh'-\frac{k}{r^2} G(h)=0
\end{equation}
which represents the harmonic map equation in the equivariant setting. We let $\bar h$ be the solution of \eqref{HM} determined by 
\beq\bar h(0)=s_0 \text{ and } \lim\limits_{r\to 0}r^{-\gamma}(\bar h(r)-s_0)=1. \label{ABhbar}\eeq

The qualitative behaviour of these solutions was described by J\"ager and Kaul in \cite{Kaul} for the special case of  corotational harmonic maps from $\R^d$ to $S^d$. Based on their methods we obtain the following result.
\begin{prop}
\label{PropKaul}
Let $N$ be a compact, rotationally symmetric manifold and let $\chi$ be a $k$-eigenmap. Given any local minimum $s_0$ of $g^2$ we let $s^\star>s_0$ be the local maximum of $g^2$ to the right of $s_0$. Then the behaviour of the solution $\bar h$ of \eqref{HM} satisfying \eqref{ABhbar} can be described as follows.
\begin{enumerate}
\item[(i)] If $-4kG'(s^\star)\leq (d-2)^2$ and if condition (C1) is satisfied, then $\bar h$ is increasing and converges to $s^\star$ as $r\to \infty$. 
\item[(ii)] Otherwise $\bar h$ still converges to a local extremum $\tilde s$ of $g^2$ (not necessarily equal to $s^\star$). The convergence is monotone if $-4kG'(s)\leq (d-2)^2$ in a neighbourhood of $\tilde s$, while $\bar h$ oscillates around the level $s=\tilde s$ infinitely many times if $-4kG'(\tilde s)> (d-2)^2$.
\end{enumerate}
\end{prop}

Noticing that the rescaled solution $H_a(r) := h_a(a^{-1/\gamma}r)$ solves
$$
\left\{
\begin{array}{l} H_a''(r) + \left( \frac{d-1}{r} + \frac{r}{2a^{1/\gamma}}\right) H'_a(r) - \frac{k}{r^2}G(H_a(r)) = 0\\
\lim_{r\rightarrow 0} r^{-\gamma} (H_a(r) - s_0) = 1 \end{array} \right.,
$$
we obtain by continuous dependence of solutions of differential equations on the coefficients the following lemma.

\begin{lemma}
\label{lemma7}
Let $(h_a)$ be the family of solutions to equation \eqref{EX} constructed in Proposition \ref{Prop6} and let $\bar h$  be the solution of \eqref{HM} satisfying \eqref{ABhbar}. 
\begin{itemize}
\item[(i)] Given any numbers $R_0>0$ and $\eps>0$, there exists $a_0>0$ such that
$$\sup\limits_{a\geq a_0}\sup\limits_{r\in [0,R_0]} \abs{h_a(a^{-1/\gamma}r)-\bar h(r)}<\eps.$$  
\item[(ii)] The map $\R\ni a\mapsto h_a(R)$ is continuous for every $R\in[0,\infty]$. 
\end{itemize}
\end{lemma}
Here we write for short $h_a(\infty)$ for the limit $\lim_{r\to \infty}h_a(r)$ which exists according to Lemma \ref{lemma2}. For the proof of these results we refer once more to \cite{Dis}.

\subsection{Comparison principles} \label{section:comp}
Comparison principles and maximum principles are very valuable tools to analyse the behaviour of solutions of differential equations. 
To study the properties of solutions of equation \eqref{EX} for general settings, we use
\begin{lemma}
\label{lemma4}
Let $G\in C^1((0,\infty))$ and $\phi\in C((0,\infty))$ be arbitrary fixed functions. We consider the differential operator 
\begin{equation}
\label{Tphi}
T_\phi (f):=f''+(\frac{d-1}{r}+\phi)f'-\frac{k}{r^2}\cdot G(f)
\end{equation}
on an interval $I=[r_1,r_2]\subset(0,\infty)$.
\begin{itemize}
\item[(i)] Suppose that $G\vert_{(a,b)}> 0$ on some interval $(a,b)\subset \R$. Then a non-constant function $f\in C^2(I, (a,b))$ with $T_\phi(f)\geq 0$ cannot achieve a local maximum in the interior of $I$.
\item[(ii)] Suppose that $G'\vert_{(c,d)}> 0$ on some interval $(c,d)\subset \R$. Let $f_1\neq f_2$ be two functions in $C^2(I,(c,d))$ with 
$$T_\phi(f_2)\leq T_\phi(f_1) \text{ on } I.$$ 
Assume that
$$c< f_2(r_1)\leq f_1(r_1)<d\text{ and } f_2'(r_1)\leq f_1'(r_1).$$
Then
$$f_2(r)<f_1(r) \text{ and } f_2'(r)<f_1'(r)$$
for all $r\in I$.
\end{itemize}
\end{lemma}
This lemma can be easily reduced to the classical maximum principle by the use of Taylor expansion, see the proof of Proposition \ref{CompPrinc} below.

We remark that the condition $G'> 0$ is violated for the non-linearity $G=g\cdot g'$ of the equations \eqref{HM} and \eqref{EX} in a neighbourhood of $s^\star$ if $C_{s^\star}$ is an equator of $N$. Using the above lemma, we can thus compare solutions of these equations only as long as they map into an appropriate neighbourhood of a pole or a minimal sphere. In contrast, the following comparison principle applies to general solutions of \eqref{EX} if the considered setting satisfies the assumptions of Theorem \ref{THM1} (ii).
\begin{prop}[comparison principle]
\label{CompPrinc}
Let $k\in\N$, let $s_1<s_2$ and let $G\in C^1(\R)$ be any given function. Assume that 
$$4k\theta\leq (d-2)^2 \;\;\;\;\mbox{with}\;\;\;\;\theta:=\max\{-G'(s): s\in [s_1,s_2]\},
$$
and furthermore that $\phi(r)\geq c\cdot r$ for a constant $c=c(\phi)>0$.
Then the following comparison principle holds true for the operator 
$T_\phi$ defined by \eqref{Tphi}. 

Let $h_1$ and $h_2$ be two functions in $C^2((0,\infty),[s_1,s_2])$ such that 
$$T_\phi(h_1)\geq T_\phi(h_2)$$
and assume that 
\begin{equation}
\label{r0}
h_1(r_0)\geq h_2(r_0)\quad\text{ and } \quad h_1'(r_0)\geq h_2'(r_0)
\end{equation} for some $r_0> 0$.
Then either $h_1$ and $h_2$ coincide or 
\begin{enumerate}
\item[(i)] $\qquad \displaystyle{h_1(r)>h_2(r) \text{ for all } r>r_0}$ \\ $ $
\item[and] $ $ \\
\item[(ii)] $\qquad \displaystyle{\lim\limits_{r\to \infty}h_1(r)>\lim\limits_{r\to \infty}h_2(r)}.$
\end{enumerate}
\end{prop} 
\begin{rem} 
By the characterisation of energy-minimising equator maps given in Proposition \ref{Prop3.1} the above comparison principle applies in particular to all solutions of \eqref{EX} if the setting $(N,\chi)$ satisfies the assumptions of Theorem \ref{THM1} (ii).
\end{rem}

\begin{proof}[Proof of Proposition \ref{CompPrinc}]
Let $h_1$ and $h_2$ be as in the statement of Proposition \ref{CompPrinc} and assume that $h_1\neq h_2$. In order to prove statement (i),
we consider the rescaled difference
$$f_1(r):=r^\eta\cdot(h_1(r)-h_2(r))$$
for $\eta>0$ to be determined later. Observe that $f_1$ satisfies the linear differential inequality
\begin{equation*}
 f_1''+(\frac{d-1-2\eta}{r}+\phi)f_1'+a_\eta(r) f_1\geq0
\end{equation*}
for 
\begin{equation*}%\label{aeta}
a_\eta(r)=\frac{\eta(\eta+1)}{r^2}-\frac{\eta}{r}(\frac{d-1}{r}+\phi)-\frac{k\cdot G'(\xi)}{r^2}< \frac1{r^2}[\eta^2-(d-2)\eta+k\theta]. \end{equation*}
Choosing $\eta=\frac{d-2}{2}$ in view of our assumption that $4k\theta\leq (d-2)^2$ we have $a_\eta<0$. Thus, if we assume that $f_1$ achieves a positive local maximum at a point $r_1\geq r_0$ a contradiction results; hence $f_1$ is an increasing, positive function on $[r_0,\infty)$ and statement (i) follows.\\

For the second part of the proof we let $\eta=\frac{d-2}{2}$ be as above and consider 
$$f_2(r):=\big(\frac{r}{C+r}\big)^\eta\cdot(h_1(r)-h_2(r))$$
for a (large) constant $C$ which is chosen later on. 

The first part of the proof implies that if $r_0<1$ then  
$$f_1'(1)=\eta(h_1-h_2)(1)+(h_1-h_2)'(1)\geq \delta$$ for some $\delta=\delta(h_1,h_2)>0$. For $f_2$ defined as above, we thus find not only that $f_2(1)\geq 0$, but also that
\begin{align*}
 f_2'(1)&=\,\big(\frac{1}{C+1}\big)^\eta\cdot (\frac{C\eta}{C+1}(h_1-h_2)(1)-(h_1-h_2)'(1))\\
&\geq\, \big(\frac{1}{C+1}\big)^\eta\cdot (\frac{C\delta}{C+1}-\frac{1}{C+1}\eta(h_1-h_2)(1))\geq 0 
\end{align*}
for $C$ sufficiently large. We can thus assume that $f_2(r_0)\geq 0$ and $f_2'(r_0)\geq 0$ for some $r_0\geq 1$.
The function $f_2$ satisfies the inequality
\begin{equation*}
 %\label{f2}
f_2''+(\frac{d-1-2\eta}{r}+\frac{2\eta}{C+r}+\phi)f_2'+\tilde a_C(r) f_2\geq 0
\end{equation*}
where the coefficient $\tilde a_C(r)$ may be estimated as
$$\tilde a_C(r)\leq %\frac1{r^2}[\eta^2-(d-2)\eta+k\theta] 
\frac{(d-3)\eta}{r(C+r)}+(\eta-\eta^2)\frac{2C+r}{(C+r)^2r}-\frac{\phi(r)}{r}\frac{C\eta}{C+r}.$$
On the interval $[1,\infty)$ the dominating term in the above bound is $-\frac{\phi(r)}{r}\frac{C\eta}{C+r}<0$ and thus $\tilde a_C(r)< 0$ if $C$ is large enough.
The same argument as above implies that $f_2$ is increasing and positive on $[r_0,\infty)$. Therefore
$$\lim\limits_{r\to \infty}h_1(r)-h_2(r)=\lim\limits_{r\to \infty}f_2(r)>0$$
as claimed. 
\end{proof}

\section{The ODE approach: proof of Theorem \ref{THM1}} \label{section:THM1}

\subsection{Proof of (i): existence of selfsimilar solutions}
We begin with the proof of the existence statement.
So let $s\in\R$ be any given number. If $s$ is a local extremum of $g^2$ then the constant function $h_s\equiv s$ induces a selfsimilar solution to \eqref{HF} for initial data $u_0(x)=(s,\chi(\frac{x}{\abs{x}}))$.
By symmetry we may thus assume that $G(s)>0$ and we denote by $s_0<s<s^\star$ the local minimum respectively local maximum of $g^2$ to the left respectively right of $s$.

The image by the continuous function 
$$L:a\mapsto \lim_{r\to\infty} h_a(r)$$
of $[0,\infty)$ is an interval, see Lemma \ref{lemma7}; we claim that it
contains the interval $[s_0,s^\star)$. Thus choosing $a_s>0$ such that $L(a_s)=s$, we obtain that $u(x,t):=R_\chi h_{a_s}$ is a solution of \eqref{HF} for the considered initial data $u_0(x)=(s,\chi(\frac{x}{\abs{x}}))$.

We first prove the corresponding claim for the continuous function 
$$M:[0,\infty) \ni a\mapsto \sup_{r\in \R} h_a(r).$$

Let us first remark that since $M(0)=s_0$, it is enough to show that to any given $\eps>0$ there is a number $a=a(\eps)>0$ such that $M(a)>s^\star-\eps$.

Let $\bar h$ be the solution of \eqref{HM} satisfying \eqref{ABhbar}. By Proposition \ref{PropKaul} the function $\bar h$ converges to a local extremum of $g^2$ as $r\to \infty$. Since the quantity $V$ defined in \eqref{V} is decreasing also for solutions of \eqref{HM}, we find that 
\beq
\label{V2}
-k g^2(s_0)=V(0)>V(r)\geq -k g^2(\bar h(r))
\eeq
for every $r>0$ and thus that $\bar h(r)>s_0$ for every $r>0$.

Consequently $\lim_{r\to \infty}\bar h(r)\geq s^\star$ and given any $\eps>0$ we may choose $R>0$ with $\bar h(R)>s^\star -\eps/2$. Lemma \ref{lemma7} then implies that $M(a)\geq \bar h_a(Ra^{-1/\gamma})\geq s^\star-\eps$ for $a$ large enough. This establishes the claim that $[s_0,s^\star)\subset M([0,\infty))$.
\\

It is now crucial to remark that $h_a$ is increasing if $M(a)<s^\star$ according to Lemma \ref{lemma4}. We thus find that the $L(a)=M(a)$ for these values of $a$ and the claim $[s_0,s^\star)\subset L([0,\infty))$ follows.

This concludes the proof of the existence statement of Theorem \ref{THM1}.

\subsection{Proof of (iii): multiplicity of solutions}

We now give the proof of the non-uniqueness result stated in Theorem \ref{THM1} (iii). So let $C_{s^\star}$ be an equator of a rotationally symmetric manifold such that the equator map $u^\star_{s^\star,\chi}$ is not even locally energy minimising, i.e. such that $-4kG'(s^\star)>(d-2)^2$.
Let $s_0<s^\star<s_1$ be the local minima of $g^2$ to the left and to the right of $s^\star$. We can assume without loss of generality that $g^2(s_0)\geq g^2(s_1)$.

Let $(h_a)_{a\geq 0}$ be the family of solutions to \eqref{EX} with $h_a(0)=s_0$ constructed in Proposition \ref{Prop6} and let $\bar h$ be the solution to equation \eqref{HM} satisfying \eqref{ABhbar}. Since the inequality \eqref{V2} is valid also for the functions $h_a$ we find that $s_0< h_a, \bar h<s_1$ on $(0,\infty)$ for each $a>0$. According to Proposition \ref{PropKaul} the function $\bar h$ thus converges to $s^\star$ as $r\to \infty$ while oscillating around the level $s=s^\star$ infinitely many times. 

We consider now the function 
\begin{equation}
\label{s.I}
[0,\infty)\ni a\mapsto I(a):=\#\{r>0:\, h_a(r)=s^\star\}
\end{equation}
counting the number of intersection points of the function $h_a$ with the level $s=s^\star$ of the equator $C_{s^\star}$. 

Lemma \ref{lemma7}, Lemma \ref{lemma2} and the above remark imply that $I(a)=I(0)=0$ for $a>0$ small enough while $I(a)\to \infty$ as $a\to \infty$.

The number $I(a)$ is however finite for each $a\in [0,\infty)$; in fact, we prove
\vspace{0.2cm}
\begin{lemma}\label{lemmaR}
For any rotationally symmetric manifold $N$, any equator $C_{s^\star}$ of $N$ and for any $k\in\N$ there exists a number $R>0$ such that the following holds true.
\begin{enumerate}
 \item[(i)] No solution $h$ of \eqref{EX} intersects the level $s=s^\star$ more than once on the interval $[R,\infty)$.
 \item[(ii)] If $h(r)=s^\star$ for some $r>R$, then $h$ cannot converge to $s^\star$ as $r\to \infty$.
\end{enumerate}
\end{lemma}
\vspace{0.2cm}
\begin{proof}
The key idea is to compare a given solution $h$ of \eqref{EX} with supersolutions of an appropriate differential equation for which the comparison principle is valid. 
So let $N$ be any rotationally symmetric manifold, let $C_{s^\star}$ be an equator of $N$ and let $k\in\N$.

We set $\Theta:=\max_{s\in\R}-G'(s)$ for the function $G=g\cdot g'$ and choose $D\geq d$ such that 
$$4k\Theta\leq (D-2)^2.$$
We claim that Lemma \ref{lemmaR} holds true for $R:=2\sqrt{D-d}.$

So let $h$ be a solution of \eqref{EX} with $h(r)=s^\star$ for some $r\geq R$. By symmetry we can assume that $h'(r)<0$. The claim is obviously true if $h$ is decreasing on all of $[r,\infty)$.
Suppose therefore that $h$ achieves a local minimum at some point $(r_0,h(r_0))$, $r_0>R$.

We now consider the solution $f$ of
\begin{equation}
\label{eqf}
f''+\frac{D-1}{r} f'-\frac{k}{r^2}G(h)=0
\end{equation}
with $f(0)=s_0$ and $\lim_{r\to 0}r^{-\Gamma}(f(r)-s_0)=1$, for $\Gamma:=\frac12(\sqrt{(D-2)^2+4kG'(s_0)}-(D-2))>0$. As usual, $s_0<s^\star$ denotes the local minimum of $g^2$ to the left of $s^\star$.

We should remark here that \eqref{eqf} does not necessarily represent the harmonic map equation in a new geometric setting since $k$ is in general no eigenvalue of $\Delta_{\mathbb{S}^{D-1}}$. Nonetheless the existence of $f$ still follows from standard methods. Furthermore the characterisation of solutions given by Proposition \ref{PropKaul} remains valid for equation \eqref{eqf}. 
The solutions $f_a(r)=f(a^{1/\Gamma}r)$, $a>0$ of \eqref{eqf} are thus increasing on $(0,\infty)$ and converge to $s^\star$ as $r\to \infty$. Since $h(r_0)<s^\star$ we find  
$$h(r_0)<f_a(r_0) \, \text{ and } h'(r_0)=0<f_a'(r_0)$$
for $a$ large enough.

Since $f_a$ is an increasing solution of \eqref{eqf}, it satisfies $\widetilde T_{r/4}(f_a)\geq 0$ on all of $(0,\infty)$ for the operator 
$$\widetilde T_{r/4}(f):=f''+(\frac{D-1}{r}+\frac{r}4) f'-\frac{k}{r^2}G(f).$$
On the other hand, let $r_1\in (r_0,\infty]$ be the maximal number such that $h$ is increasing on $(r_0,r_1)$. By our choice of $R$ and the assumption that $r_0>R$ we then find that $\widetilde T_{r/4}(h)\leq 0$ on $(r_0,r_1)$. Since the operator $\widetilde T_{r/4}$ satisfies the assumptions of the comparison principle, we find 
$$h\leq f_a<s^\star \text{ on } (r_0,r_1).$$
However, according to Lemma \ref{lemma4} the function $h$ cannot achieve a local maximum at $r_1$ unless $h(r_1)>s^\star$. Therefore $r_1= \infty$ and $h<s^\star$ on $(r_0,\infty)$. Finally, the comparison principle implies $\lim_{r\to \infty} h(r)<\lim_{r\to \infty} f_a(r)=s^\star.$
\end{proof}
The connection between the properties of the function $I(\cdot)$ defined in \eqref{s.I} and the existence of multiple solutions to the initial value problem \eqref{HF}, \eqref{IV} is given by
\begin{lemma}
\label{intersect}
The function $I:[0,\infty)\to \N_0$ defined in \eqref{s.I} has the following properties if $N$, $\chi$ and $C_{s^\star}$ satisfy the assumptions of Theorem \ref{THM1} (iii).
\begin{enumerate}
\item[(i)] $I$ is subcontinuous on $[0,\infty)$\footnote{i.e. for every $a_0\in [0,\infty)$ and every sequence $a_n\to a_0$ we have $I(a_0)\leq \underline{\lim}_{n\to \infty} I(a_n)$} and if $a_0$ is a point of discontinuity of $I(\cdot)$ then
$$ \underset{a\to a_0}{\ov{\lim}} I(a)=I(a_0)+1$$
and 
$$\lim_{r\to \infty}h_{a_0}(r)=s^\star.$$
\item[(ii)] For any $n\in \N_0$ there is number $A_n>0$ with $I(A_n)=n$ such that the corresponding solution $h_{A_n}$ of \eqref{EX} converges to $s^\star$ as $r\to \infty$.
\item[(iii)] The union $S_{2k}\cup S_{2k+1}$ of the sets 
$$S_n:=\{\lim_{r\to \infty} h_a(r):\, I(a)=n\},\quad n\in\N_0,$$
is a neighbourhood of $s^\star$ for every $k\in \N_0$.
\end{enumerate}
\end{lemma}
As an immediate consequence of this lemma, we obtain the third statement of Theorem \ref{THM1} for the neighbourhoods $U_K$ of $s^\star$ given by
$$U_K:=\bigcap_{n=0}^{K-1} (S_{2n}\cup S_{2n+1}).$$ 

\begin{proof}[Proof of Lemma \ref{intersect}]
We need to understand how the number of intersection points of the continuous family of maps $(h_a)$ with the level $s=s^\star$ can change as we vary the parameter $a$. So let $a_0\in[0,\infty)$ be any given number.
Let us first remark that no solution of \eqref{EX} can be tangential to the level $s=s^\star$ of the equator at any point $r>0$; this follows from the definition of an equator as a local maximum of $g^2$ and since the quantity $V$ introduced in \eqref{V} is decreasing.
In addition $h_{a_0}(0)\neq s^\star$ and we therefore find a neighbourhood of $a_0>0$ on which $I(\cdot)\geq I(a_0)$. In particular $I$ is subcontinuous at each point.

Let us now assume that $a_0$ is a point of discontinuity of $I(\cdot)$ and let $a_i\to a_0$ be such that $\lim_{i\to \infty} I(a_i)=\ov{\lim}_{a\to a_0} I(a)>I(a_0)$. Let $R>0$ be the number determined in Lemma \ref{lemmaR} and recall that at most one of the zeros of $h_{a_i}-s^\star$ can be larger than $R$. In addition we can check that 
$$\norm{h_{a_0}-h_{a_i}}_{C^1([0,2R])}\underset{i\to\infty}{\longrightarrow} 0,$$ 
compare with Lemma \ref{lemma7} the corresponding remarks.
If the distance between two distinct roots of $h_{a_i}$ were to converge to zero as $i\to \infty$ we would therefore find a point $0\leq r<R$ with $h_{a_0}(r)=s^\star$ and $h_{a_0}'(r)=0$. As remarked before this is impossible. 

The discontinuity of $I$ at $a_0$ must therefore be caused by roots of $h_{a_i}-s^\star$ \textit{escaping to infinity} in the sense that $h_{a_i}(r_i)=s^\star$ for a sequence $r_i\to \infty$ as $i\to \infty$. 

By Lemma \ref{lemmaR} all roots of $h_{a_i}-s^\star$ different from $r_i$ must be strictly less than the constant $R$ for $i$ large enough. Consequently $I(a_i)\leq I(a_0)+1$ for $i$ large.

Furthermore, Lemma \ref{lemma2} implies that 
$$\abs{\lim_{r\to \infty}h_{a_i}(r)-s^\star}=\abs{\lim_{r\to \infty}h_{a_i}(r)-h_{a_i}(r_i)}\leq \frac{\ov{C}}{2r_i^2}\underset{i\to \infty}{\longrightarrow} 0. $$
Applying Lemma \ref{lemma7} we find that $h_{a_0}$ converges to $s^\star$ as $r\to \infty$ as claimed in (i).

A first consequence of statement (i) and the fact that $I(a)\to \infty$ as $a\to \infty$ is that $I:[0,\infty)\to \N_0$ is surjective. Given any number $n\in \N_0$ we can thus define 
$$A_n:=\max\{a:\, I(a)=n\}\in (0,\infty).$$ The function $I$ is obviously discontinuous at $A_n$ and we conclude that $h_{A_{n}}$ tends to $s^\star$ as $r\to \infty$ by statement (i).

Finally, according to the first part of the proof, we can choose $\eps_n>0$ so small that the solutions $h_a$ intersect the level $s=s^\star$ at a point $r_a>R$ for all $a\in (A_n,A_n+\eps_n)$. Lemma \ref{lemmaR} thus implies that $\lim_{r\to \infty}h_a(r)\neq s^\star$ for all $a\in (A_n,A_n+\eps_n)$. But of course 
$$\lim_{r\to \infty}h_a(r)\underset{ a\to A_n}{\longrightarrow} s^\star=\lim_{r\to \infty}h_{A_n}$$ again by Lemma \ref{lemma7}.

The connected subset 
 $$\{\lim_{r\to \infty}h_a(r):\, a\in(A_{n-1},A_{n-1}+\eps_{n-1})\}\subset I_{n},\quad n\in\N$$ therefore contains an open interval of the form $(s^\star-\delta_n,s^\star)$ (for $n$ even) respectively $(s^\star,s^\star+\delta_n)$ (for $n$ odd). Since $I_0=[s_0,s^\star]$ the final claim of Lemma \ref{intersect} follows. 
\end{proof}

\subsection{Proof of (ii): uniqueness of solutions}

Finally we turn to the proof of the uniqueness result stated in Theorem \ref{THM1} (ii). We first show 
\begin{lemma}
\label{lemma5.1}
Let $N$ be rotationally symmetric and let $\chi$ be a $k$-eigenmap. Let $s_0$ be a local minimum of $g^2$ for which condition $(C2)$ holds true and let $(h_a)$ be the family of solutions to \eqref{EX} with $h_a(0)=s_0$ constructed in Proposition \ref{Prop6}. Assume that condition (C1) holds true for the local maximum $s^\star>s_0$ of $g^2$ to the right of $s_0$ and that $-4kG'(s^\star)\leq (d-2)^2$.
Then the map
$$L:a\mapsto \lim\limits_{r\to \infty}h_a(r)$$
is a continuous bijection from $[0,\infty)$ to $[s_0,s^\star)$.
\end{lemma}

\begin{proof}[Proof of Lemma \ref{lemma5.1}] We observe first of all that $h_a(r)\geq s_0$ for every $a\geq 0$ and every $r\geq 0$ since inequality \eqref{V2} is valid also for solutions of \eqref{EX}. Recall now that the solution $\bar h$ of the harmonic map equation \eqref{HM} to initial data \eqref{ABhbar} is increasing on $[0,\infty)$ with $\lim_{r\to \infty} \bar h(r)=s^\star$, see Proposition \ref{PropKaul}. The rescaled functions $\bar h_a(r):=\bar h(ra^{1/\gamma})$ are thus supersolutions of \eqref{EX} and we find that $h_a\leq\bar h_a<s^\star$ by the comparison principle. We may furthermore apply the comparison principle to conclude that two different solutions $h_a$ and $h_{\tilde a}$ do not intersect at any finite $r>0$ nor converge to the same limit as $r\to \infty$. Thus $L$ is increasing and Lemma \ref{lemma5.1} follows since we have already shown that $[s_0,s^\star)\subset L([0,\infty))$. 
\end{proof}
\begin{rem}
The above proof shows in particular that the solutions $h_a$ never reach the level $s=s^\star$ of the equator and thus that they are increasing by Lemma \ref{lemma4}.
\end{rem}
Let now $N$ and $\chi$ be as in Theorem \ref{THM1} (ii). 

Given any number $s\in\R$ we let $s_1^\star\leq s<s^\star_2$ be the local maxima of $g^2$ to the left and right of $s$ and $s_0$ the local minimum of $g^2$ in $(s_1^\star,s_2^\star)$. We then need to show that the only solution of \eqref{EX} in $H^1_{rad}(\R^d)$ with limit $s$ is given by $h_{L^{-1}(s)}$ for the family $(h_a)_{a\in[-\infty,\infty]}$ of solutions to \eqref{EX} constructed in Proposition \ref{Prop6} with $h_a(0)=s_0$. Here $L$ stands for the bijection $L:\R\cup\{\pm\infty\}\to [s^\star_1,s_2^\star]$ of Lemma \ref{lemma5.1} which we extend by $L(-\infty):=s_1^\star$ and  $L(\infty):=s_2^\star$. Furthermore, we denote by $h_{-\infty}\equiv s_1^\star$ and $h_\infty\equiv  s_2^\star$ the constant solutions of \eqref{EX} which induce the corresponding equator maps.

Since by assumption \textit{all} equator maps are $\chi$-energy-minimising Proposition \ref{Prop3.1} implies that 
$-4kG'(s^\star)\leq (d-2)^2$ for every equator $C_{s^\star}$ of $N$. By Lemma \ref{lemma5.1} we thus know that the above solution $u_s$ is unique among all solutions to \eqref{HF}, \eqref{IV} induced by elements of the families $(h_a)$, $h_a(0)$ any local minimum of $g^2$, of Proposition \ref{Prop6}. \\

To conclude the proof of Theorem \ref{THM1}, we therefore only need to show that there are no selfsimilar, equivariant solutions to \eqref{HF}, \eqref{IV} other than those induced by these families $(h_a)$ of solutions to \eqref{EX}. This is achieved in the following proposition which is valid for arbitrary compact manifolds and eigenmaps $\chi$.
\begin{prop}\label{total}
Let $N$ be any compact, rotationally symmetric manifold, $\chi$ a $k$-eigenmap and assume that condition (C2) is valid. 
Then every solution $h\in H^1_{rad}(\R^d)$ to \eqref{EX} is a member of one of the families $(h_a)_{-\infty\leq a\leq \infty}$ given in Proposition \ref{Prop6} corresponding to the local minima of $g^2$.
\end{prop}

This result might be surprising since the condition imposed by $h\in H_{rad}^1(\R^d)$ is relatively mild. A priori, it does not exclude functions with singularities at $r=0$, but merely restricts the allowed blow-up rates. 
\\

As we will see below, most solutions of equation \eqref{EX} are unbounded and can thus be described by
\begin{lemma}
\label{lemma5.2} Let $N$ be compact, $k\in\N$ and let $h$ be an unbounded solution of equation \eqref{EX}. Then there exist $\delta>0$ and $\eps>0$ such that 
\begin{equation*}
%\label{L5.2}
\abs{h'(r)}>\frac{\delta}{r^{d-1}}\text{ for } r\in(0,\eps).
\end{equation*}
In particular $h\notin H_{rad}^1(\R^d)$.
\end{lemma}
\begin{proof}
Let $h$ be any unbounded solution of \eqref{EX}. Since $N$ is compact, $h$ must reach the level of a pole for some $r_0>0$, i.e.~$g(h(r_0))=0$. 
Let now 
$$\widetilde V(r)=\widetilde V(h)(r):=r^{2(d-1)}\cdot \big[\abs{h'}^2-\frac{k}{r^2}g^2(h)\big].$$
Obviously $\widetilde V(r_0)\geq 0$ and a short calculation shows that $\widetilde V$ is decreasing for any non-constant solution of \eqref{EX}. Given any $0<\eps<r_0$, we can thus choose $\delta>0$ such that 
$\widetilde V\vert_{[0,\eps]}\geq \delta^2>0$ and the claim follows.
\end{proof}

\begin{comment}
\begin{rem}
\label{rem2} The preceding proof in particular shows that solutions of \eqref{EX} for which $\widetilde V(r)\geq 0$ for some $r>0$ are unbounded and not contained in $H^1_{rad}(\R^d)$. An important consequence is that no bounded solution of \eqref{EX} can reach the level of a pole, be it for finite $r>0$ or as $r\to \infty$, unless it is constant.
We would like to add that the reasoning above also applies to solutions of the harmonic map equation \eqref{HM}.
\end{rem}
\end{comment}

The behaviour of general solutions to \eqref{EX} is furthermore restricted by
\begin{lemma}
\label{lemma5.3} Let $N$ and $\chi$ be as in Proposition \ref{total}. Then for any solution $h$ of \eqref{EX} there exists $\eps=\eps(h)>0$ such that $h\vert_{(0,\eps)}$ is monotonous.
\end{lemma}
\begin{proof} For simplicity we give the details of the proof only for settings satisfying the assumptions of Theorem \ref{THM1} (ii). In this case we can show the stronger result that solutions of \eqref{EX} achieve at most one local extremum on all of $(0,\infty)$.

So let $N$ and $\chi$ be as in Theorem \ref{THM1} (ii) and let $h$ be a solution of \eqref{EX} that attains a local extremum, say a local minimum, at some point $(r_0,h(r_0))$, $r_0>0$. Then Lemma \ref{lemma4} tells us $G(h(r_0))>0$. We denote by $s_0<h(r_0)<s^\star$ the local minimum respectively the local maximum of $g^2$ to the left respectively right of $h(r_0)$. Let now $(h_a)$ be the family of solutions to \eqref{EX} with $h(0)=s_0$. The functions $h_a$ are increasing for every $a>0$ and $h_a(r_0)$ tends to $s^\star$ as $a\to \infty$. Choosing $a>0$ large enough, we thus have $s_0<h(r_0)<h_a(r_0)$ and $h'(r_0)=0<h_a'(r_0)$. By the comparison principle we conclude that 
$$h(r)<h_a(r)<s^\star \text{ for all } r\geq r_0.$$
Therefore $h$ we cannot achieve any local maximum and thus any local extremum at all after $r_0$ according to Lemma \ref{lemma4}

The claim follows because $r_0$ was chosen as an arbitrary extremal point of $h$.
\end{proof}

\begin{rem}
The proof of Lemma \ref{lemma5.3} for general settings makes use of the fact that the decreasing quantity $V$ of \eqref{V} is negative for bounded solutions of \eqref{EX} and satisfies
$$V(r_1)-V(r_2)>\Delta$$ for all local extrema $0<r_1<r_2<1$ of $h$ and a constant $\Delta(h)>0$; for details we refer to \cite{Dis}. 
\end{rem}

Finally, we conclude the proof of our main uniqueness result for selfsimilar solutions, Theorem \ref{THM1} (ii), by giving the 
\begin{proof}[Proof of Proposition \ref{total}]
Let $h\in H^1_{rad}(\R^d)$ be any solution of \eqref{EX}. By Lemma \ref{lemma5.2} the function $h$ is bounded. It can therefore be extended continuously up to $r=0$ according to Lemma \ref{lemma5.3}. We analyse the properties of $h$ based on the value $h(0)=\lim_{r\to 0} h(r)$. We begin with 
\\

\textbf{Case 1.} $h(0)$ is a local minimum of $g^2$.

Let $s_0$ be any local minimum of $g^2$ and let $\gamma>0$ and $(h_a)_{a\in \R}$ be as in Proposition \ref{Prop6}. We know that any solution $h$ of \eqref{EX} with $h(0)=s_0$ and $\lim_{r\to 0}r^{-\gamma}(h(r)-s_0)=a\in \R$ coincides with  $h_a$ by the uniqueness statement of Proposition \ref{Prop6}. 

So let us assume that there exists a solution $h$ of \eqref{EX} with $h(0)=s_0$ for which $r^{-\gamma}(h(r)-s_0)$ diverges as $r\to 0$.
According to Lemma \ref{lemma5.3} and by symmetry, we may assume that $h$ is increasing on a small interval $(0,\eps)$. 
We chose $b>s_0$ such that $G'\vert_{[s_0,b]}>0$ and fix $r_0\in (0,\eps)$ with $h(r_0)<b$. Following the arguments of the proof of statement (i) of Theorem \ref{THM1} we then find $a_0>0$ with $h(r_0)<h_{a_0}(r_0)$ and with $h_{a_0}\vert_{[0,r_0]}\leq b$. According to Lemma \ref{lemma4} the function $h_{a_0}$ is an upper bound for $h$ on $[0,r_0]$ and thus $\overline{\lim}_{r\to 0}r^{-\gamma}(h(r)-s_0)\leq a_0<\infty$. Since this quantity by assumption diverges, there exists a number $a>0$ with 
$$ 0\leq \underset{r\to 0}{\underline\lim}r^{-\gamma}(h(r)-s_0)< a<\underset{r\to 0}{\ov\lim}r^{-\gamma}(h(r)-s_0).$$

But then $h$ has to intersect the corresponding solution $h_a$ of \eqref{EX} in points arbitrarily close to $r=0$ in contradiction to 
Lemma \ref{lemma4}.

We conclude that the only solutions of \eqref{EX} with $h(0)=s_0$ are those of the family $(h_a)_{a\in\R}$.
\\

\textbf{Case 2.}\, $h(0)$ is a local maximum of $g^2$.

Let $C_{s^\star}$ be an equator of $N$. We claim that the only solution of \eqref{EX} with $h(0)=s^{\star}$ is the constant map $h_{\infty}\equiv s^\star$.

Indeed, let us assume that $h$ is a non-constant solution of \eqref{EX} with $h(0)=s^\star$ and let $r_1>0$ be such that $g^ 2(h(r_1))<g^2(s^\star)$. We set $\delta:=g^2(s^\star)-g^2(h(r_1))>0$ and choose $r_0\in(0,r_1)$ such that $g^2(h(r))\geq g^2(s^\star)-\delta/2$ for all $r\in [0,r_0]$. Since the quantity $V(r)$ given by \eqref{V} is non-increasing we obtain that on $(0,r_0)$
$$(rh')^2-kg^2(s^\star)+k\delta/2\geq V(r)\geq V(r_1)\geq-kg^2(s^\star)+k\delta.$$
Consequently $$\abs{h'(r)}\geq \frac{\sqrt{k\delta/2}}{r}$$ on $(0,r_0)$ and $h$ cannot converge as $r\to 0$, in contradiction to the assumption $h(0)=s^\star$.
\\

Finally, we need to consider
\\

\textbf{Case 3.} $h(0)$ is no local extremum of $g^2$.

We have assumed from the very beginning that $g'$ has no roots of multiplicity greater than one and thus find that $G(h(0))\neq 0$. By symmetry we can focus on solutions $h$ of \eqref{EX} with $G(h(0))>0$. 

Suppose $h$ is decreasing on some interval $(0,\eps)$. We can then bound the second derivative of $h$ on a small interval $(0,r_0]\subset(0,\eps)$ by
$$h''=k\cdot\frac{G(h(0))+o(1)}{r^2}-\big(\frac{d-1}{r}+\frac{r}{2}\big)h'\geq \frac{c}{r^2}$$
for a constant $c>0$ independent of $r$ and for $o(1)\to 0$ as $r\to 0$.

Integrating the obtained inequality from $r$ to $r_0$ gives
$$h'(r)\leq-\frac{c}{r}+h'(r_0)+\frac{c}{r_0}=-\frac{c}{r}+C(r_0)$$
for every $r\in(0,r_0)$, which is obviously wrong for bounded functions $h$.

According to Lemma \ref{lemma5.3}, we obtain that $h$ is increasing on some interval $(0,\eps)$. Using the divergence form of \eqref{EX} given in \eqref{div-form} we then find for $r\in(0,r_0)$
$$(e^{r^2/4}r^{d-1}h')'\geq cr^{d-3}$$
for a constant $c>0$ and for $r_0>0$ small enough.

Integrating from $r/2$ to $r<r_0$ we find
$$e^{r^2/4}r^{d-1}h'(r)\geq \Big(\frac{r}{2}\Big)^{d-1}e^{r^2/16}h'\Big(\frac{r}{2}\Big)+c\frac{1-2^{2-d}}{d-2}r^{d-2}\geq \tilde c r^{d-2}>0,$$
since $h$ is increasing on $(0,\eps)$.
The resulting lower bound of $h'(r)\geq \frac{\tilde c}{r}$ on $(0,r_0)$ once more  stands in contrast to the assumption that $h$ is continuous up to $r=0$.
\\

We conclude that $h(0)$ is a local extremum of $g^2$ for each bounded solution $h$ of \eqref{EX}. Combined with cases 1 and 2 and the description of unbounded solutions 
of Lemma \ref{lemma5.2}, we obtain Proposition \ref{total}.
\end{proof}
This concludes the proof of Theorem \ref{THM1}.

\section{Stability from time $t=0$: proof of Theorem~\ref{THM3}}
We study the stability properties of the constant in time solutions $u(x,t)=R_\chi(s^\star)$ of the harmonic map flow. That is to say, we consider the Cauchy problem 
\begin{equation}
\label{perturb}
\left\{
\begin{array}{l}
f_t-f_{rr}-\frac{d-1}{r}f_r+\frac{k}{r^2}\left[G(f+s^\star)-G(s^\star)\right]=0 \\ f(t=0)=f_0,
\end{array} 
\right.
\end{equation}
where $s^\star$ is such that $G(s^\star)=0$.

\subsection{Proof of (i): linear stability}

The linearised version of the above equation is obviously
\begin{equation*}
f_t-[\Delta f-\frac{k G'(s^\star)}{r^2} f] = 0.
\end{equation*}
Here and in the following $\Delta$ denotes the radial Laplacian on $\R^d$, $\Delta f:=f_{rr}+\frac{d-1}{r}f_r$.
By Hardy's inequality~\eqref{Hardy}, the operator $\displaystyle -\Delta + \frac{c}{r^2}$ is positive on $L^2$, or $H^1$, if $\displaystyle c > -\frac{(d-2)^2}{4}$. This suffices to prove (i).

\subsection{Proof of (ii): weak solutions} $ $

\noindent \underline{The a priori estimate:}
Let us begin with a formal derivation of the a priori estimate on $f$, solving~(\ref{perturb}), which is at the heart of the proof of $(ii)$. 
Since by assumption $k \inf G' > -\frac{(d-2)^2}{4}$, Taylor's formula gives
$$
\frac{k}{r^2}\left[G(f+s^\star)-G(s^\star)\right] \geq \left( -\frac{(d-2)^2}{4} + \epsilon \right) f 
$$
for some $\epsilon>0$. Thus, taking the scalar product of~\eqref{perturb} with $f$ in space, and integrating in time gives
$$
\displaystyle \|f\|_{L^\infty ([0,\infty),L^2(\mathbb{R}^d))}^2 + \|\nabla f\|_{L^2 ([0,\infty), L^2(\mathbb{R}^d))}^2 \leq C \|f_0\|_2^2,
$$
for some constant $C$ by the same argument as in (i). 

\bigskip

\noindent \underline{The rigorous proof:} In order to turn the above a priori estimate into a rigorous proof, we make use of an approximation scheme. Let $\chi$ be a smooth function, zero in a neighbourhood of the origin, and equal to one outside a (larger) bounded neighbourhood of the origin.
Then let $f^\epsilon$ solve
\begin{equation*}
\left\{
\begin{array}{l}
f^\epsilon_t-f^\epsilon_{rr}-\frac{d-1}{r}f^\epsilon_r+\frac{k}{r^2}\chi \left(\frac{r}{\epsilon} \right) \left[G(f^\epsilon+s^\star)-G(s^\star)\right]=0 \\ f^\epsilon(t=0)=f_0
\end{array} 
\right.
\end{equation*}
It is clear that for $\epsilon>0$, the above equation has a unique solution $f^\epsilon$ in $L^\infty L^2 \cap L^2 \dot{H}^1$, which is, by the above estimate, uniformly bounded in this space. Furthermore,
$$
\left\| f^\epsilon_t \right\|_{\dot{H}^{-1}(\mathbb{R}^d)} \leq \left\| \Delta f^\epsilon \right\|_{\dot{H}^{-1}(\mathbb{R}^d)} + \left\| \frac{G(f^\epsilon+s^\star)-G(s^\star)}{r^2} \right\|_{\dot{H}^{-1}(\mathbb{R}^d)}.
$$
Arguing by duality and using Hardy's inequality gives
\begin{equation*}
\begin{split}
\left\| \frac{G(f^\epsilon+s^\star)-G(s^\star)}{r^2} \right\|_{\dot{H}^{-1}} & = \sup_{\|\phi\|_{\dot{H}^1}\leq 1} \left| \int \frac{G(f^\epsilon+s^\star)-G(s^\star)}{r^2} \phi \right| \\
& \lesssim \sup_{\|\phi\|_{\dot{H}^1}\leq 1} \int \frac{|f^\epsilon|}{r} \frac{|\phi|}{r} \\
& \lesssim \sup_{\|\phi\|_{\dot{H}^1}\leq 1} \|f^\epsilon\|_{\dot{H}^1} \|\phi\|_{\dot{H}^1} = \|f^\epsilon\|_{\dot{H}^1}.
\end{split}
\end{equation*}
Putting together the two above inequalities gives
$$
\left\| f^\epsilon_t \right\|_{L^2 ([0,\infty),\dot{H}^{-1}(\mathbb{R}^d))} \lesssim \left\| f^\epsilon \right\|_{L^2 ([0,\infty),\dot{H}^{1}(\mathbb{R}^d)}
$$
which implies a uniform bound for $f^\epsilon_t$ in $L^2 \dot{H}^{-1}$. 
By Aubin's lemma (see for instance~\cite{Showalter}), the set of functions which is bounded in $L^2 H^1$, with time derivatives bounded in $L^2 H^{-1}$, embeds compactly in $L^2 L^{\frac{2d}{d-2}-\delta}_{loc}$. Thus a subsequence of $f^\epsilon$ converges to a function $f$ in $L^2 L^{\frac{2d}{d-2}-\delta}_{loc}$, where $\delta$ is positive and small.

We can now pass to the limit in the equation. The linear terms are of course easily handled. As for the nonlinear term, the strong convergence of $f_\epsilon$ implies that
$$
G(f^\epsilon+s^\star)-G(s^\star) \rightarrow G(f+s^\star)-G(s^\star) \;\;\;\;\mbox{in $L^2 L^{\frac{2d}{d-2}-\delta}_{loc}$}.
$$
On the other hand, $\frac{1}{r^2}\chi \left(\frac{r}{\epsilon} \right)$ converges strongly in $L^\infty L^{\frac{d}{2}-\delta}_{loc}$ to $\frac{1}{r^2}$ as $\epsilon$ goes to zero. Thus, 
$$
\frac{1}{r^2}\chi \left(\frac{r}{\epsilon} \right)\left[G(f^\epsilon+s^\star)-G(s^\star)\right] \rightarrow \frac{1}{r^2}\left[G(f^\epsilon+s^\star)-G(s^\star)\right]\;\;\;\;\mbox{in $L^2 L^{\frac{2d}{d+2}-\delta}_{loc}$}
$$
(for a new choice of $\delta$), which concludes the proof.

\subsection{Proof of (iii): strong solutions if $G'(s^\star)>0$} 
\label{grenouille}
Let $s^\star$ be such that $G(s^\star)=0$ and $G'(s^\star)>0$.
We first need to introduce some new notations: set
$$
c:= k G'(s^\star)>0\;\;\;\;H_c := -\Delta + \frac{c}{r^2} \;\;\;\;\mbox{and} \;\;\;\; J(x) := \frac{k G(x+s^\star) - c x}{x^2}
$$
(so that $J$ is a smooth and bounded function). This turns~\eqref{perturb} into
\begin{equation}
\label{eqH}
\left\{ \begin{array}{l}
\displaystyle f_t + H_c f = \frac{f^2}{r^2} J(f) \\ f(t=0) = f_0. \end{array} \right.
\end{equation}

The necessary estimates will be provided by the following lemma:

\begin{lemma}
\label{estim}
If $r^2 F \in L^\infty_t L^\infty_x$ and $f_0 \in L^\infty$, there exists a unique solution in $L^\infty_t L^\infty_x$ to
\begin{equation}
\label{eqH1}
\left\{ \begin{array}{l}
\displaystyle f_t + H_c f = F \\ f(t=0) = f_0. \end{array} \right.
\end{equation}
Furthermore, it satisfies
$$ \|f\|_{L^\infty_t L^\infty_x} \lesssim \|f_0\|_{L^\infty_x} + \|r^2 F\|_{L^\infty_t L^\infty_x}.$$
\end{lemma}

With the help of this lemma, it is easy to solve~(\ref{eqH}) by Picard's fixed point theorem: rewrite~(\ref{eqH}) via Duhamel's formula as
$$
f(t) = e^{-tH_c} f_0 + \int_0^t e^{(s-t)H_c} \frac{f^2}{r^2} J(f)(s)\,ds := \operatorname{RHS}(f).
$$
Lemma~\ref{estim} easily gives the estimates
$$
\left\| \operatorname{RHS}(f) \right\|_{L^\infty_t L^\infty_x} \lesssim \left\| f_0 \right\|_{L^\infty} + \|f\|_{L^\infty_t L^\infty_x}^2
$$
and
$$
\left\| \operatorname{RHS}(f) - \operatorname{RHS}(\widetilde{f}) \right\|_{L^\infty_t L^\infty_x} \lesssim \max (\|f\|_{L^\infty_t L^\infty_x},\|\widetilde{f}\|_{L^\infty_t L^\infty_x}) \|f - \widetilde{f} \|_{L^\infty_t L^\infty_x}.
$$
Thus the map $RHS$ is a contraction on a small enough ball in $L^\infty_t L^\infty_x$ which implies the existence of a unique fixed point and thus of a solution of \eqref{eqH} in this small ball.

\begin{proof}[Proof of Lemma~\ref{estim}] 1. The uniqueness part follows by the maximum principle (see for instance Quittner and Souplet~\cite{QS}, Prop.~52.4, page 509).

\medskip

2. Assuming a priori the existence of a solution $f$ in $L^\infty L^\infty$ to~(\ref{eqH1}), let us prove the bounds. If $F = 0$, they follow since the kernel of $e^{-tH_c}$ is positive, and pointwise smaller than the kernel of $e^{t\Delta}$, as is easily checked. Suppose now that $f_0=0$; by positivity of the kernel of $e^{-tH_c}$, it suffices to consider the case $F\geq 0$. Observe that
$$
H_c \frac{1}{c} = \frac{1}{r^2}.
$$
Thus $\frac{1}{c}$ is a constant solution of $\widetilde{f}_t + H_c \widetilde{f} = \frac{1}{r^2}$, and $\frac{\|r^2 F\|_{L^\infty_x L^\infty_t}}{c}$ a supersolution for our problem. By the maximum principle, 
$$
\|f\|_{L^\infty_t L^\infty_x} \lesssim \frac{\|r^2 F\|_{L^\infty_x L^\infty_t}}{c}.
$$

\medskip

3. It is now standard to obtain the existence result by combining these a priori bounds with an approximation scheme.
\end{proof}

\subsection{Optimality of Theorem~\ref{THM3}} \label{instab} We discuss for the statements (i) and (ii) to what extent they are optimal; in other words, for both of these statements we examine whether the given sufficient condition is also necessary.

\bigskip \noindent 
\underline{Statement (i)} The assumption is clearly optimal, since for $c<-\frac{(d-2)^2}{4}$, any self-adjoint extension of $-\Delta + \frac{c}{r^2}$ has an unbounded spectrum.

\bigskip \noindent \underline{Statement (ii)} This statement becomes wrong if $G'(s^\star) < -\frac{(d-2)^2}{4}$.
This corresponds to the situation where $C_{s^\star}$ is an equator which is not locally energy-minimising. 

As we saw in the proof of Theorem~\ref{THM1}, there exist profiles $\psi$, such that $h(r,t) = \psi\left(\frac{r}{\sqrt{t}}\right)$ is a solution to the equivariant harmonic 
map heat flow
$$h_t-h_{rr}-\frac{d-1}{r}h_r+\frac{k}{r^2}G(h)=0.$$
with $h(t=0,r) = \lim_{r\rightarrow \infty} \psi = s^\star$, and $\psi \not \equiv s^\star$. 
Consider the data $h(r,\epsilon)$ for such a solution $h$ and let $u$ be the corresponding solution of the harmonic map flow. By taking $\epsilon$ small, this data can be made arbitrarily close to $s^\star$ in $L^2(\mathbb{R}^d)$: 
denoting $f=u-s^\star$ the difference, this means $\|f(t=0)\|_{L^2}$ arbitrarily small.
However, the $L^2(\mathbb{R}^d)$ distance between $h(r,t+ \epsilon)$ and $s^\star$, $\|f(t)\|_2$, goes to infinity as $t$ goes to infinity: This contradicts $(ii)$ since $f(r,t+\epsilon)$ is the only solution to
$$
f_t-f_{rr}-\frac{d-1}{r}f_r+\frac{k}{r^2}\left[G(f)-G(s^\star)\right] = 0
$$
associated to the initial data $f(r,\epsilon)$.
Indeed, this solution is smooth and decays fast (as can be verified since $\psi$ converges to the equator), and thus one can easily prove ``weak -strong uniqueness'': any other solution satisfying the energy inequality~(\ref{tomahawk}) has to agree with this one.

%% file: part3.tex
\section{Stability from time $t=1$: proofs of theorems~\ref{chien} and~\ref{thm:lin-stab}}

In this section, the problem will be analysed in self-similar variables
$$\sigma:=\log(t) \quad \rho=\frac r{\sqrt{t}}.$$
Setting
$$w(\rho,\sigma) = v(e^{\sigma/2}\rho,e^\sigma) = v(r,t).$$
equation~(\ref{equiv}) becomes
\begin{equation}
\label{flower}
\partial_\sigma w - \partial_\rho^2 w - \left( \frac{d-1}{\rho} + \frac{\rho}{2} \right) \partial_\rho w + \frac{k}{\rho^2} (gg')(w) = 0.
\end{equation}
As we shall see, the operator in $\rho$ can be made self-adjoint in $L^2(d\mu)$ with
$$
d\mu(\rho)=e^{\abs{\rho}^2/4} \rho^{d-1} d\rho.
$$

\subsection{Proof of Theorem~\ref{chien}}

Equation~(\ref{flower}) can also be written
$$
\partial_\sigma w - \rho^{1-d} e^{-\frac{\rho^2}{4}} \partial_\rho \left[ \rho^{d-1} e^{\frac{\rho^2}{4}} \partial_\rho w \right] + \frac{k}{\rho^2} G(w) = 0
$$
Taking the scalar product with $\partial_\sigma w$ in $L^2 (d\mu)$ yields
$$
\partial_\sigma \overline{ E} (w) = - \left\| \partial_\sigma w \right\|_{L^2(d\mu)}^2,
$$
which gives the desired result.

\subsection{Proof of Theorem~\ref{thm:lin-stab}}

In self-similar coordinates the linearised equation of the harmonic map flow \eqref{LIN} reads
\begin{equation} \label{SS}
\partial_\sigma w- \partial_\rho^2 w-(\frac{d-1}{\rho} + \frac{\rho}{2}) \partial_\rho w+\frac{k}{\rho^2}G'(\psi(\rho))w=0.
\end{equation}

Recalling $d\mu(\rho)=e^{\abs{\rho}^2/4} \rho^{d-1} d\rho$, consider the weighted space
$$H:=L_{rad}^2(\mu):=\{v:\R\to \R: \int \abs{v}^2\,d\mu<\infty\}$$ and the corresponding Sobolev spaces. This is of course equivalent to considering the radial elements of $L^2(\mathbb{R}^d,e^{|y|^2/4}dy)$.
The operator
$$Aw:=-\partial_\rho^2 w -(\frac{d-1}{\rho}+\frac{\rho}{2})\partial_\rho w+\frac{k}{\rho^2}G'(\psi(\rho))w$$
defined on the dense subspace $H_{rad}^2(\mu)$ of the Hilbert space $(H,\norm{\cdot}_H)$ is symmetric.
The main step for the proof of Theorem \ref{thm:lin-stab} is to show
\begin{prop}
\label{Prop:EV}
The operator $A$ has a selfadjoint extension onto a dense subspace of $H$ whose spectrum is discrete. Furthermore, the number of eigenvalues less than one is equal to the number of local extrema of the function $\psi$ representing the original selfsimilar solution $u(x,t)=R_\chi \psi(\frac{x}{\sqrt{t}})$.
\end{prop}

Theorem~\ref{thm:lin-stab} immediately follows from this proposition by transforming back to the original coordinates.

\begin{proof}[Proof of Proposition \ref{Prop:EV}]
Let us first remark that the operator $-\Delta -\frac{\rho}{2}\partial_{\rho} $ is non-negative since 
$$\langle (-\Delta -\frac{\rho}{2}\partial_{\rho} )w,w\rangle_{L^2(\mu)} =\int \abs{w'}^2 d\mu$$
for every $w\in H_{rad}^2(\mu)\subset L^2_{rad}(\mu)$. 

Since $N$ is compact and smooth the function $G'$ is bounded from below. Recall furthermore that $\psi(0)$ is the coordinate 
of either a pole or a minimal sphere by Proposition \ref{total}. We thus find an interval $[0,R_0]$ on which $G'(\psi(\cdot))>0$. 

The multiplication operator
$w\mapsto \frac{k}{\rho^2}G'(\psi(\rho))w$ is thus bounded from below in $H$ by some constant $\gamma$. Consequently, the same holds true for the operator $A$, i.e.~we have that $$\langle A w,w\rangle_H \geq\gamma\norm{w}_H^2$$ for every $w\in H^2_{rad}(\mu)$.

By the Friedrich's extension theorem the operator $A$ thus has a unique selfadjoint extension (still denoted by $A$) onto a domain $\mathcal{D}(A)\subset L^2_{rad}(\mu)$ contained in the form domain of $A$, i.e.~in $H^1_{rad}(\mu)$.\\

We now analyse the spectrum of this selfadjoint operator and begin by showing that it is discrete. 

Let $R_0$ be as above and let $\phi\in C_c^\infty([0,\infty),[0,1])$ be such that $\text{supp}(\phi)\subset [0,R_0]$ and 
$\phi\equiv 1$ on $[0,R_0/2]$.

We decompose the operator $A$ as
$$A=A_0+A_1$$ for the bounded multiplication operator $A_1: L^2(\mu)\to L^2(\mu)$ given by
$$A_1w:=(1-\phi)\cdot \frac{k}{\rho^2}G'(\psi(\rho))w.$$ 
We show
\begin{lemma}
\label{lemma:komp}
The operator $A_0:=A-A_1:\mathcal{D}(A)\to L^2(\mu)$
is a bijective unbounded operator with compact inverse.
\end{lemma}

\begin{proof}
We consider the bilinear form 
$$B(w_1,w_2):=\langle A_0w_1,w_2\rangle_{L^2(\mu)}$$
induced by $A_0$. By the choice of $\phi$, the definition of $A$ and Hardy's inequality \eqref{Hardy2} we can extend $B(\cdot,\cdot)$ 
to a bounded and coercive bilinear form on all of $H_{rad}^1(\mu)$. The representation theorem of Riesz then implies 
the existence of an isomorphism $L$ from the dual space $(H^1_{rad}(\mu))^*$ to $H_{rad}^1(\mu)$ such that 
$$B(Lf,w)=\langle f,w\rangle$$
for every linear form $f\in (H_{rad}^1(\mu))^*$.

Remark that by definition $A_0$ and $L^{-1}$ agree on $\mathcal{D}(A)$ and that the domain $\mathcal{D}(A)$ is nothing else than the image of 
$L_{rad}^2(\mu)\subset (H_{rad}^1(\mu))^*$ under $L$ by the maximality of selfadjoint operators. Thus $A:\mathcal{D}(A)\to L^2(_{rad}\mu)$ is a bijection 
with inverse given by 
$$A^{-1}=\iota\circ L\vert_{L^2_{rad}(\mu)}.$$
Here $\iota:H^1_{rad}(\mu)\to L_{rad}^2(\mu)$ denotes the inclusion map. Contrary to the inclusion maps of standard Sobolev spaces on $\R^d$, 
the map $\iota$ is compact. In fact, the compactness of this operator can be easily derived from the inequality
$$\int \rho^2 w^2 d\mu\leq 16\int \abs{w'}^2 d\mu$$
which follows from 
\begin{align*}
 0&\leq d\int w^2 d\mu=\int \frac{d}{d\rho}(\rho^d)e^{\rho^2/4}w^2 \,d\rho=-\int(w^2e^{\rho^2/4})'\rho^{d}\,d\rho\\
&= -2\int ww'\rho\, d\mu-\frac12\int w^2\rho^2\, d\mu\\
&\leq 2\big(\int w^2 \rho^2\,d\mu\big)^{1/2}\cdot\big(\int\abs{w'}^2\, d\mu\big)^{1/2}-\frac12\int w^2 \rho^2 \, d\mu,
\end{align*}
compare also \cite{Vazquez-Zuazua}.
The lemma follows since the inclusion map $L^2_{rad}(\mu)\hookrightarrow (H^1_{rad}(\mu))^*$ is of course continuous.
\end{proof}
As an immediate consequence of the above lemma we obtain that the spectrum of $A_0^{-1}$ contains at most countably many eigenvalues
 which cannot accumulate at any point different from zero. Therefore the spectrum of $A_0$ is discrete.
Finally we need to remark that since $A_1$ is bounded, Lemma \ref{lemma:komp} implies that $A_1$ is relatively compact with respect to $A_0$. 
Thus the essential spectra of $A=A_0+A_1$ and $A_0$ agree and are thus empty, see e.g. \cite{Reed-Simon}. \\

To establish Proposition \ref{Prop:EV} we need to analyse the individual eigenvalues. 

If $\lambda\in\R$ is an eigenvalue of $A$ and if $v_\lambda\in L_{rad}^2(\mu)$ is a corresponding eigenfunction then $v_\lambda$ solves the equation
\begin{equation}
\label{EF}
-v''-(\frac{d-1}{\rho}+\frac{\rho}{2})v'+\frac{k}{\rho^2}G'(\psi(\rho))v=E\cdot v
\end{equation}
for $E=\lambda$. 

The asymptotic behaviour of solutions of the above linear differential equation can be described by
\begin{lemma}
\label{asymptEF}
Let $\psi\in H^1_{rad}(\R^d)$ be any solution of \eqref{EX} and let $E\in \R$. 
\begin{enumerate}
 \item[(i)] Let  $\gamma_1<0<\gamma_2$ be the solutions of the equation $\gamma^2+(d-2)\gamma-kG'(\psi(0))=0$. 
Then there are solutions $\phi_1$ and $\phi_2$ of \eqref{EF} such that 
$$\lim_{\rho\to 0} (\phi_i(\rho)+\rho\cdot\phi_i'(\rho))\rho^{-\gamma_i}=1, \quad i=1,2.$$
\item[(ii)] There are solutions $\phi_3$ and $\phi_4$ of \eqref{EX} such that 
\begin{equation} \label{limit1}
\lim_{\rho\to \infty} (\abs{\phi_3(\rho)}+\abs{\rho^{-1}\cdot\phi_3'(\rho)})e^{\rho^2/4}\rho^{d-2E}=1
\end{equation}
respectively
\begin{equation}\label{limit2}
\lim_{\rho\to \infty} (\abs{\phi_4(\rho)}+\abs{\rho^{-1}\cdot\phi_4'(\rho)})\rho^{2E}=1
\end{equation}
\end{enumerate}
\end{lemma}

One way to prove the above lemma is to study the asymptotics as $s\to \infty$ of the functions $s\mapsto (v(e^{-s}), \frac{d}{ds} v(e^{-s}))$, respectively of $s\mapsto (v(\sqrt{s}),  \frac{d}{ds} v(\sqrt{s})$. One can check that each of these functions satisfies a system of first order differential equations for which Theorem 8.1 of \cite{Cod-Lev} applies. The claimed asymptotics follow. 

Given any $E\in \R$ we let $v_E$ be the solution of \eqref{EF} that satisfies 
$$v_E(0)=0 \text{ and } \lim_{\rho\to 0}\rho^{-\gamma_2}(v_E(\rho)+\rho\cdot v_E'(\rho))=1$$
where $\gamma_2>0$ is the constant determined in Lemma \ref{asymptEF}.

Let us remark that $v_E$ is in general not an element of $H$. However, if $E$ is an eigenvalue of $A$ then $v_E\in H$ must be (a multiple of) the corresponding eigenmap since other solutions of \eqref{LIN} are not square integrable (with respect to $\mu$) near the origin and thus certainly not in $H$. 
The multiplicity of each eigenvalue is thus one. 

Furthermore we have the following connection between the properties of the solutions $v_E$, $E\in\R$, and the eigenvalues of $A$.
\begin{lemma}
\label{lemma:Sturm-Liou}
For every $E_0\in \R$ the number of eigenvalues 
$$n_{E_0}:=\#\{E<E_0:\, E\text{ eigenvalue of } A\}$$ that are less than $E_0$ 
coincides with the number 
$$N_{E_0}:=\#\{\rho>0:\, v_{E_0}(\rho)=0\}$$
of zeros of the function $v_{E_0}$ on $(0,\infty)$.
\end{lemma}

\begin{proof}
We use methods known from the theory of Sturm-Liouville operators as presented in chapter XIII.3 of \cite{Reed-Simon}.

Let us first recall that $A$ is bounded from below and that the eigenvalues 
$$\lambda_1<\lambda_2<...$$
are discrete and have multiplicity one. 

Let now $E_0\in\R$ be any fixed number and let $N_{E_0}$ and $n_{E_0}$ be defined as above. We first show that 
$$\lambda_{N_{E_0}}<E_0.$$
We denote by 
$$0=\rho_0<\rho_1<...<\rho_{N_{E_0}}$$
the zeros of $v_{E_0}$. It may now be easily checked that the functions 
$$\mathds{1}_{(\rho_{i-1},\rho_i)}\cdot v_{E_0}, \quad i=1,..,N_{E_0}$$
span a $N_{E_0}$ dimensional subspace of the form domain $H^1_{rad}(\mu)$ of $A$ on which 
$$\langle A v,v\rangle_{H}\geq E_0\norm{v}_{H}^2.$$
Consequently we find that $\lambda_{N_{E_0}}\leq E_0$.

Since the function $E\mapsto N_{E_0}$ is subcontinuous (compare e.g.~Lemma \ref{intersect}) we find that also the strict inequality $\lambda_{N{E_0}}<E_0$ is valid and thus that $$n_{E_0}\geq N_{E_0}.$$
On the other hand we show\\

\textbf{Claim:} The map $E\mapsto N_E$ is non-decreasing and if $E_0$ is an eigenvalue then $N_E\geq N_{E_0}+1$ for every $E>E_0$.

Remark that since all eigenvalues of $A$ have multiplicity one, this claim implies that
$$N_{E_0}\geq n_{E_0}$$ and thus concludes the proof of Lemma \ref{lemma:Sturm-Liou}.

\textit{Proof of claim.}
Let $E_0\in\R$ and let $0=\rho_0<\rho_1<..<\rho_{N_{E_0}}$ be the zeros of $v_{E_0}$. We show on the one hand that all functions $v_E$, $E>E_0$, have a zero in each interval $(\rho_{i-1},\rho_i)$. On the other hand we prove that if $E_0$ is an eigenvalue of $A$ then there is a further zero of $v_E$ in the interval $(\rho_{N_{E_0}},\infty)$. 

We begin by the proof of this second claim. So let $E_0$ be an eigenvalue of $A$ and assume that there exists some $E>E_0$ such that $v_E$ has no zero in $(\rho_{N_{E_0}},\infty)$. By symmetry it is enough to consider the case that $v_E>0$ and $v_{E_0}>0$ in this interval. We now consider the integral
$$I:=\int_{\rho_{N_{E_0}}}^\infty \frac{d}{d\rho}\big[(v_E\cdot v_{E_0}'- v_E'\cdot v_{E_0})e^{\rho^2/4}\rho^{d-1}\big] \,d\rho.$$
Since $v_{E_0}$ is an eigenfunction of $A$ and thus an element of $L^2(\mu)$, Lemma \ref{asymptEF} implies that for $\rho\geq 1$
$$\abs{v_{E_0}(\rho)}+\abs{\rho^{-1}v_{E_0}'(\rho)}\leq C\cdot e^{-\rho^2/4}\cdot \rho^{2E_0-d}.$$
Conversely all solutions of \eqref{EF}, and thus is particular $v_E$, are bounded by
$$\abs{v_{E}(\rho)}+\abs{\rho^{-1}v_{E}'(\rho)}\leq C\cdot \rho^{-2E}$$
as $\rho\to\infty$.
We thus find that
$$\abs{(v_E(\rho)v_{E_0}'(\rho)-v_E'(\rho)v_{E_0}(\rho))e^{\rho^2/4}\rho^{d-1}}\leq C\rho^{2(E_0-E)}\underset{\rho\to \infty}{\to} 0$$
and therefore
$$I=-v_E(\rho_{N_{E_0}})v_{E_0}'(\rho_{N_{E_0}})\exp(\rho_{N_{E_0}}^2/4)\rho_{N_{E_0}}^{d-1}<0.$$
On the other hand $v_E$ is a solution of \eqref{EF} and thus a short calculation shows that 
$$
I=\int_{\rho_{N_{E_0}}}^\infty (E-E_0)v_E v_{E_0} d\mu>0$$
which leads to a contradiction.

The same argument applied on the intervals $(\rho_{i-1},\rho_i)$, $i=1,..,N_{E_0}-1$, shows that $v_E$ has a zero in the intervals $(\rho_{i-1},\rho_i)$. The assumption that $E_0$ is an eigenvalue is not needed for this part of the proof since we are integrating over compact intervals and thus do not need to control the asymptotics of $v_{E_0}$ as $\rho\to \infty$.

This concludes the proof of the above claim and thus of Lemma \ref{lemma:Sturm-Liou}.
\end{proof}

In order to establish Proposition \ref{Prop:EV} we finally need to understand how the solution $v_1$ of \eqref{EF} for $E=1$ is connected with the function $\psi$ representing the original selfsimilar solution of the harmonic map flow. 

Since the harmonic map flow is invariant under translations, the maps 
$$u_\eps(x,t):=u(x,t+\eps)=R_\chi \psi(\frac{x}{\sqrt{t+\eps}})$$
are solutions of \eqref{HF} on $\R^d\times (-\eps,\infty)$ for every $\eps\in\R$. 

Consequently 
$$\left.\frac{d}{d\eps}\big(\psi(\frac{x}{\sqrt{t+\eps}})\big) \right\vert_{\eps=0}=-\frac1{2t}\cdot\frac{x}{\sqrt{t}} \psi(\frac{x}{\sqrt{t}})$$
solves the linearised equation \eqref{LIN}. Working in selfsimilar coordinates, we thus find that the function 
$$\rho\mapsto \rho\cdot \psi'(\rho)$$ solves equation \eqref{EF} for $E=1$.

Since $\rho\psi'(\rho)=O(\rho^{\gamma_2})$ as $\rho\to 0$ for the constant $\gamma_2>0$ of Lemma \ref{asymptEF} the function $v_1$ is equal to (a multiple of) $r\psi'(r)$. The number of zeros of $v_1$ is thus given by the number of local extrema of $\psi$.

This concludes the proof of Proposition \ref{Prop:EV} and thus of our final result Theorem \ref{thm:lin-stab}.
\end{proof}
Finally we would like to remark that the number of local extrema of a solution \eqref{EX} coincides with the number of times this solution intersects the level of the equator. Thus a selfsimilar solution of the harmonic map flow enjoys the stability property of Theorem \ref{thm:lin-stab} (i), if and only if it does not cross the equator.